\documentclass[11pt]{article}

\usepackage{calc}
\usepackage{amssymb,amsmath,amsfonts,amsthm}
\usepackage{latexsym}
\usepackage{graphics}
\usepackage{indentfirst}
\usepackage{hyperref}
\usepackage{comment}
\usepackage[shortlabels]{enumitem}
\usepackage[english]{babel}
\usepackage{tikz}
\usepackage{esdiff}
\usepackage[english]{babel}
\usepackage{amsthm}
\usepackage{amssymb}
\usepackage[utf8]{inputenc}
\usepackage[english]{babel}
\usepackage{amsmath}
\usepackage{times}
\usepackage{graphicx}
\usepackage{float}
\usepackage[margin=1in]{geometry}
\usepackage{blindtext}
\usepackage{amssymb}
\usepackage{amsthm}
\usepackage{mathtools}
\usepackage{textcmds}
\usepackage{comment}

\usepackage{comment}
\usepackage[mathlines]{lineno}
\usepackage[utf8]{inputenc}
\usepackage{standalone}
\usepackage{thmtools} 

\usepackage{fancyhdr}

\newtheorem{internalcustomfact}{Fact}
\newenvironment{customfact}[1]
{\renewcommand\theinternalcustomfact{#1}\internalcustomfact}
{\endinternalcustomfact}

\newtheorem*{thm*}{Theorem}
\newtheorem{thm}{Theorem}
\newtheorem{lem}[thm]{Lemma}

\newtheorem{obs}[thm]{Observation}
\newtheorem{cor}[thm]{Corollary}
\newtheorem{conj}[thm]{Conjecture}
\newtheorem{cl}{Claim}[thm]

\newtheorem*{definition*}{Definition}
\newtheorem{definition}{Definition}
\newtheorem{hypothesis}[thm]{Hypothesis}
\newtheorem{fact}[thm]{Fact}
\theoremstyle{remark}
\newtheorem{remark}{Remark}[section]
\newtheorem{example}{Example}

\newcommand{\x}{\mathrm{x}}

\newcommand{\ex}{\mathrm{ex}}
\newcommand{\EX}{\mathrm{EX}}
\newcommand{\spex}{\mathrm{spex}}
\newcommand{\SPEX}{\mathrm{SPEX}}
\newcommand{\supnorm}[1]{\left\lVert{#1}\right\rVert_\infty}

\newcommand{\tmld}{\mathcal{T}_{m,l+1}^{\delta}}

\newcommand{\dmoore}{\left((\delta - 1)^2 + 1\right)}

\newcommand{\comments}[1]{}

\allowdisplaybreaks

\makeatletter
\newcommand{\specificthanks}[1]{\@fnsymbol{#1}}
\makeatother

\begin{document}
	
	\title{A Spectral Tur\'an Problem for a Fixed Tree}
	
	\author{Dheer Noal Desai\thanks{Department of Mathematical Sciences, The University of Memphis, Memphis, TN 38152. E-mail: {\tt dndesai@memphis.edu}} \and Hemanshu Kaul \thanks{Department of Applied Mathematics, Illinois Institute of Technology, Chicago, IL 60616. E-mail: {\tt kaul@iit.edu}} \and Bahareh Kudarzi\thanks{Department of Applied Mathematics, Illinois Institute of Technology, Chicago, IL 60616. E-mail: {\tt bkudarzi@hawk.iit.edu}} }

	\maketitle
	\begin{abstract}
		We study the spectral Tur\' an problem for trees. To avoid limiting our perspective to specific families of trees, we parametrize trees in terms of their unique bipartition. We say $T \in \tmld$ if $T$ is a tree of order $m$, where the order of the smaller partite set $A$ of $T$ is $l+1$, and $\delta$ is the minimum degree of the vertices in $A$. The motivation for this parametrization comes from the recent proof of the spectral Erd\H{o}s-S\' os conjecture. For a given fixed tree $T$, we describe $\SPEX(n,T)$ and consequently, bound $\spex(n,T)$ in terms of $m,l,\delta$ for that tree. Our approach combines spectral arguments with new results and constructions on embedding a tree $T \in \tmld$ into graphs of the form  $\overline{K}_l  \vee m K_{1,\delta-1}$, that might be of independent interest.  We give an almost complete description of the degrees of vertices in graphs from $\SPEX(n, T)$ and consequent bounds on $\spex(n,T)$ within an error of $\Theta(n^{-1/2})$ and $\Theta(n^{-1})$ that are based on our embedding results for the given $T$.
		Our results make progress on determining $\spex(n,T)$ for $\delta \ge 2$ as asked by Fang, Lin, Shu and Zhang (EJC, 2024).

		
	\end{abstract}

	\section{Introduction}\label{intro}

	A fundamental problem in extremal graph theory is to determine the Tur\' an number $\mathrm{ex}(n,F)$, the largest number of edges that an $n$-vertex graph can have without containing a subgraph isomorphic to $F$. The Erdős-Stone-Simonovits theorem \cite{ES66,ES46} offers an asymptotic formula for $\mathrm{ex}(n, F)$ when $F$ has a chromatic number of at least 3. However determining the Turán number's order of magnitude for most bipartite graphs remains a challenging problem.
	
	Within the realm of bipartite graphs, the Tur\' an number of trees can be determined within a factor of 2. Consider a tree $T$ with $k+1$ vertices. By using disjoint copies of $K_k$, we can establish that $\mathrm{ex}(n, T) \geq n\frac{(k-1)}{2}$. Conversely, a standard result in graph theory states that $\mathrm{ex}(n, T)$ is bounded above by $n(k - 1)$, that is every $n$-vertex graph with average degree more than $2(k-1)$ contains all $(k+1)$-vertex trees. In 1965, Erdős and Sós~\cite{erdos1965} conjectured that we can do better: every $n$-vertex graph with average degree more than $k-1$ contains all $(k+1)$-vertex trees. This conjecture has attracted a lot of attention in the past sixty years (e.g. see \cite{GZ, M, SW, STY, Tiner, wozniak} and \cite[pp. 224-227]{degeneratesurvey})), but it remains open in general. 
	
	In 2010, Nikiforov~\cite{Nikiforovpaths} proposed studying the spectral version of the Tur\' an problem and the Erdős and Sós conjecture. This spectral perspective falls under the setting of the Brualdi-Solheid problem \cite{bs}. For a given graph $G$, we define $\lambda=\lambda(G)$ as the \emph{spectral radius} (largest eigenvalue) of its adjacency matrix. The \emph{Brualdi-Solheid problem} asks for the maximum (or minimum) $\lambda$ within a fixed family of graphs. Over the years,  Brualdi-Solheid type problems have been widely studied and many results have been published, such as in \cite{BZ, BLL, EZ, FN, nosal1970eigenvalues, S, SAH}. Nikiforov suggested a systematic study of the Tur\' an-type variant of the Brualdi-Solheid problem: determine $\mathrm{spex}(n, F)$, that is, the maximum spectral radius within the family of $n$-vertex $F$-free graphs. Nikiforov was particularly motivated by the case where $F$ is a path and more generally the spectral version of the Erdős and Sós conjecture. 
	Much work has been done in the past 15 years on the $\mathrm{spex}(n, F)$ and $\mathrm{SPEX}(n, F)$ (the set of the $n$-vertex graphs that do not contain any copies of graph $F$ and have spectral radius equal to $\mathrm{spex}(n, F)$ for various fixed graphs $F$ as well as the variant where a family of graphs is forbidden. E.g. see \cite{byrne2024general, cioabua2022spectral, dheer, FG, evencycles, desai2024spectral, SpectralIntersectingCliques,li2024spectralEFR, li2024spectralLS, li2025spectralsupersaturation,  Yongtao21, LLT, Nikiforov07, NO, SSbook, Wilf86, YWZ, ZW, ZWF}.
	
	In addition to being interesting in its own right, studying $\mathrm{spex}(n, F)$ holds significance because the spectral radius of a graph serves as an upper bound for its average degree. Consequently, any upper bound established for $\mathrm{spex}(n, F)$ also provides an upper bound for $\mathrm{ex}(n, F)$, as depicted by the inequality:
	$$\mathrm{ex}(n, F) \leq \frac{n}{2} \cdot \mathrm{spex}(n, F).$$
	
	This inequality enhances classical results in extremal graph theory. For instance, Nikiforov's findings about $\mathrm{spex}(n, K_{r+1})$ encapsulate Turán's theorem \cite{Nikiforov07}.

	\subsection{Spectral extremal problems for trees and our motivation}
	
	Our focus is to better understand $\mathrm{spex}(n, T)$ when $T$ is a tree. Since $\lfloor \frac{n}{q}\rfloor$ disjoint copies of $K_q$ and $K_{n-\lfloor \frac{n}{q}\rfloor q}$ do not contain any tree on $q+1$ vertices and the spectral radius of this graph is $q-1$, it is evident that $\mathrm{spex}(n, T) \geq q - 1$ for a tree $T$ with $q + 1$ vertices.  Can we give better bounds when we know more about the structure of the tree $T$? Understanding $\mathrm{spex}(n, T)$ and $\mathrm{SPEX}(n, T)$ in terms of $T$ remains the overall motivation, and much progress has been made in the past 14 years.
	
	In 2010, Nikiforov~\cite{Nikiforovpaths} considered this problem systematically when $T$ is a path and made some related conjectures. For $n > k$, consider the graph $S_{n,k} = K_k \vee \overline{K}_{n-k}$~\footnote{For any graphs $G_1=(V_1,E_1)$ and $G_2=(V_2,E_2)$ with $V_1 \cap V_2 = \emptyset$, we define the join graph of $G_1$ and $G_2$ to be the graph denoted by $G_1\vee G_2=(V,E)$  where $V=V_1\cup V_2$ and $E=E_1\cup E_2 \cup \{uv: u\in V_1 \text{ and } v \in V_2 \}$.}, and let $S^+_{n,k}$ be the graph obtained from $S_{n,k}$ by adding one edge. 
	In \cite{Balister2008} the authors showed that if $n$ is sufficiently large, then among all connected graphs on $n$-vertices, $S_{n, k}$ and $S^+_{n, k}$ have the maximum number of edges while having no subgraph isomorphic to the path on $2k+2$ vertices and $2k+3$ vertices, respectively.
	Nikiforov further demonstrated that for $n$ sufficiently large, $S_{n,k}$ and $S^+_{n,k}$ have the largest spectral radius among all $n$-vertex graphs prohibiting paths with $2k + 2$ and $2k + 3$ vertices, respectively.
	 Additionally, he conjectured that these graphs represent spectral extremal cases for all trees with $2k + 2$ or $2k + 3$ vertices \cite{Nikiforovpaths}, leading to the spectral version of the Erdős-Sós conjecture. Partial results towards this conjecture for trees of diameter at most 4 were proven in \cite{hou,LBW2,LBW}.
	
	The Spectral Erdős-Sós conjecture has recently been proved by Cioab\u a, the first named author of this paper, and Tait in~\cite{dheer}.
	
	\begin{thm}
		\label{spectral erdos sos}
		Let $k \geq 2$ and $G$ be an $n$-vertex graph with sufficiently large $n$.
		\begin{enumerate}[(a)]
			\item If $\lambda(G) \geq \lambda(S_{n,k})$, then $G$ contains all trees of order $2k + 2$ unless $G = S_{n,k}$.
			\item If $\lambda(G) \geq \lambda(S^+_{n,k})$, then $G$ contains all trees of order $2k + 3$ unless $G = S^+_{n,k}$.
		\end{enumerate}
	\end{thm}
	
	With the solution of this conjecture, the question arises how we should take the next step towards understanding our motivating problem of determining $\mathrm{spex}(n, T)$ and $\mathrm{SPEX}(n, T)$ in terms of $T$. The proof and discussion in~\cite{dheer} gives us some important motivation towards how to further parameterize the family of all trees beyond the number of vertices.
	
	The concluding section of~\cite{dheer} highlights an important difference between the classical edge setting and the spectral setting if it is the case that the Erd\H{o}s-S\' os conjecture is true. If we assume the Erd\H{o}s-S\' os conjecture to hold true and $q$ divides $n$, then every fixed tree on $q+1$ vertices has the same extremal graph, namely the graph on $n$ vertices consisting of disjoint copies of $K_q$. This is not true in the spectral version. 
	For the spectral Erd\H{o}s-S\' os problem the extremal graph is $S_{n, k}$ (or $S_{n, k}^+$), and by Nikiforov \cite{Nikiforovpaths} it is also a spectral extremal graph for paths $P_{q+1}$, however $S_{n, k}$ contains several other trees on $q+1$ vertices, for example the star $K_{1, q}$. So further careful study is needed to completely understand the spectral Tur\' an number and spectral extremal graphs for other fixed trees. This type of problem is usually pursued by restricting attention to trees of a specific subfamily so that their structure can be exploited. However, we do not wish to do so. Instead we want to consider any tree but delineate certain parameters that capture the spectral extremality of any fixed tree. What parameters should we consider?
	
	A careful reading of~\cite{dheer} reveals that the structure of the spectral extremal graph is controlled by the structure of the path $P_{q+1}$. This control is two-fold. 
	Consider the case $q+1 = 2l+3$. First,  the size of the smaller partite set of $P_{q+1}$ is $l+1$ which shows up in the parametrization for $S_{n,l}^+$. 
	Second, the minimum degree $\delta$ among vertices of the smaller partite set of $P_{q+1}$ is $2$. If we consider the graph induced by the $n-l$ vertices of $S_{n,l}^+$ that are not in the clique $K_l$,
	then the maximum degree of this induced graph on $n-l$ vertices is $\delta - 1 = 2-1 = 1$, and most vertices in this induced graph have degree equal to $\delta -2$. 
	For $q+1 = 2l+2$  the size of the smaller partite set of $P_{q+1}$ is also $l+1$.  
	The minimum degree $\delta$ among vertices of the smaller partite set of $P_{q+1}$ is $1$. If we consider the graph induced by the $n-l$ vertices of $S_{n,l}$ that are not in the clique $K_l$, 
	then the induced graph on $n-l$ vertices has degree $\delta - 1 = 1-1 = 0$.
	These two observations, for odd and even paths, suggest that we should consider the order of the smaller partite set and its minimum degree $\delta$ as parameters in addition to the order of the tree.  
	
	In this paper, we will study the spectral extremal problem for trees with given order $m$ where the smaller partite set has $l+1$ vertices and the minimum degree of vertices in the smaller partite set is $\delta$. Our results will show that this parametrization of all trees gives a meaningful perspective that sheds light on the spectral extremality for any fixed tree without restricting ourselves to any specific subfamily of trees.

	\subsection{A parametrization of trees and our problem}
	
	We start with the definition of the families of trees that will be our focus in the study of the spectral extremal problem. As discussed above, we will parameterize these families in terms of the order $m$ of the tree, the order $l+1$ and the minimum degree $\delta$ of the smaller partite set of the tree.
	
	For a fixed natural number $m$, let $\mathcal{T}_m$ denote the set of all trees on $m$ vertices. Note that any tree $T \in \mathcal{T}_m$ is a bipartite graph. Let $l+1 \le m/2$ be some natural number and let 
	\[\mathcal{T}_{m,l+1} = \{T \in \mathcal{T}_m \text{ such that the smaller partite set of $T$ has } l+1 \text{ vertices}\}.\]
	
	For a fixed tree $T$ with bipartition $\{A,B\}$ where $|A| \le |B|$, let $\delta_T = \min\{d_T(v): v \in A\}$. Note that if $|A| = |B|$ then $\delta_T = 1$. Also, let
	$$\mathcal{T}_{m,l+1}^{\delta} = \{T \in \mathcal{T}_{m,l+1}\textrm{ such that } \delta_T = \delta \}.$$

	\begin{remark} \label{rem: tmld}
		Given any tree $T$, since $T$ has a unique bipartition, there are unique values of $m,l,\delta$ such that $T \in \tmld$. In this paper, we use $T \in \tmld$ as a convenient shorthand for the previous sentence.
		However, if we start from the perspective of the family $\tmld$, for $\tmld$ to be nonempty, we need to assume that $m\geq \max\{(l+1)\delta+1,2l+2\}$. For simplicity, we will not restate this condition every time we consider these families of trees. Also see Remark~\ref{rem: nonempty family}.
	\end{remark}

	\begin{remark} \label{rem: t}
		An important parameter that shows up when determining the spectral extremal graphs for trees $T \in \tmld$ is $t := m-1 - (l+1)\delta$. Each tree in $\mathcal{T}_{m,l+1}^{\delta}$ has a total of $m-1$ edges. The $(l+1)\delta$ of these edges are specified by the fact that each vertex in the smaller partite set (of size $l+1$) has degree at least $\delta$. Let $t_v:= d_T(v)-\delta$ for each vertex $v$ in the smaller partite set of $T$. We can think of $t_v$ as the ``excess" degree of the vertex $v$ beyond $\delta$. 
		Then the total number of ``excess'' edges are given by $\sum_{v}t_v=m-1-(l+1)\delta$, the defined value of $t$. Note that, for odd paths $P_{2l+3}$, we have $\delta=2$ and $t = 0$, and for even paths $P_{2l+2}$, we have $\delta =1$ and $t = l$, for all $l\in \mathbb{N}$. 
		
	\end{remark}

	In Section~\ref{sec: ex for embedding} we will discuss many examples and constructions of trees that are described in terms of the prescribed parameters $m,l,\delta$, and $t$.
	
	Before we can state our problem, we will formally define the spectral extremal number and graphs for a given family of forbidden graphs.

	For a family of graphs $\mathcal{H}$, let $\spex(n, \mathcal{H})$ denote \[\max\{\lambda(G) : |V(G)| = n,\; H \not\subseteq G \text{ for all } H \in \mathcal{H}\}.\]   
	This is called the \emph{spectral extremal number} of $\mathcal{H}$.
	
	Let $\SPEX(n, \mathcal{H})$ be the set of \emph{spectral extremal graphs} of $\mathcal{H}$, defined as \[\{G : |V(G)| = n, H \not\subseteq G \text{ for all } H \in \mathcal{H},  \text{ and } \lambda(G) = \spex(n, \mathcal{H})\}.\] 
	If $\mathcal{H} = \{H\}$, then we write $\spex(n, H)$ and $\SPEX(n, H)$ for simplicity.

	In this paper we will study the problem of determining spectral extremal number and graphs for any fixed tree $T$ in $\tmld$ for various values of $m, l$ and $\delta$. That is, our results will be on the following: $\spex(n, T)$  and $\SPEX(n, T)$ for a fixed $T$ in $\tmld$. In the next subsection, we state our results.\\
	
	\noindent \textbf{Related Work.} At the time of finalizing this paper it was brought to our attention that the paper~\cite{fang2024spectral} has also considered the parametrization with respect to $l$ and $\delta$.
	Two of their results overlap with ours and we will discuss these in the next section. However, most of our paper is focused on results beyond these. 
	
	In~\cite{fang2024spectral}, the authors also consider the parameter $\beta(T)$, the \textit{covering number} of $T$, and use it to characterize $\SPEX(n, T)$ for trees $T$ with $\delta = 1$. Here $\beta(T)$ is the size of a smallest vertex cover of $T$. Note that if $T$ has bipartition $A$, $B$ where $l+1 = |A| \le |B|$, then $\beta(T) \le l+1$. In the paper they show that $\SPEX(n, T) = K_{l, n-l}$ if and only if $\beta(T) \le l$, $\delta = 1$ and the graph induced by some covering set is $\beta(T) K_1$, while $\SPEX(n, T) = S_{n, l}$ if and only if $\beta(T) = l+1$ and $\delta = 1$. 
	The authors also add that ``it seems difficult to determine $\SPEX(n , F)$ when $\delta \ge 2$, and so we leave this as a problem." 
	Here they use $F$ to denote a tree. 
	Note that $\beta(T) = l+1$ whenever $\delta \ge 2$, and the smallest vertex cover is simply the smaller partite set of the tree. 
	In our paper we make progress on determining $\SPEX(n,T)$ for $\delta \ge 2$. 
	
	In \cite{byrne2024general}, the authors recommend studying a related extremal number $\ex^G(n, \mathcal{F})$, which is the maximum number of edges contained in an $n$-vertex $\mathcal{F}$-free graph containing $G$ as a subgraph, and their associated set of extremal graphs $\EX^G(n, \mathcal{F})$. 
	In particular, Theorem 1.1 of \cite{byrne2024general} gives several conditions on $\mathcal{F}$ for which $\SPEX(n, \mathcal{F}) \subset \EX^{K_{l, n-l}}(n, \mathcal{F})$.
	To apply it in the context of our paper, for trees $T$ in $\tmld$ where $\delta \ge 2$, we require to show the existence of some graph $H \in \EX^{K_{l, n-l}}(n, T)$ of the form $H = K_l \vee X$ where $X$ is a graph on $n-l$ vertices with at most linearly many edges in $n$, as $n \to \infty$. More crucially, we require the $l$ vertices of the smaller partite set of $K_{l, n-l}$ to induce a $K_l$ in $H$.
	Our results in this paper apply to trees for which the structure of $\EX^{K_{l, n-l}}(n, T)$ is not yet clear.
	Since there is still much to understand about this problem, it may be possible that there exist trees $T$ in $\tmld$ for some $\delta \ge 2$ with respect to which there exist no $H \in \SPEX(n, T)$ of the form $K_l \vee X$. 
	On the other hand, we prove in Theorem~\ref{thm: T-spex tighter bounds version} that there are several trees $T$ for which $\SPEX(n, T)$ consists of graphs of the form $K_l \vee X$ for some graph $X$ on $n-l$ vertices and linearly many edges in $n$.
	Consequently, it may be possible that $\SPEX(n, T) \not\subset \EX^{K_{l, n-l}}(n, T)$ for some $T$.


	\subsection{Outline and discussion of our results}
	\label{sec: outline}

	In order to study $\spex(n, T)$  and $\SPEX(n, T)$ for a given tree $T$ in $\tmld$, we need to first establish conditions on $m,l,\delta$ that allow us to embed~\footnote{Given two graphs $H$ and $G$, we will say that $H$ can be embedded in $G$ if there is a copy of $H$ contained as a subgraph in $G$.} $T$ into a graph that gives us insight into the spectral extremal graphs of $T$ and whose spectral radius is easier to study while still close to $\spex(n, T)$. 
	We study embeddings of $T$ in $\tmld$ into graphs of the form $\overline{K_l} \vee m S_{\delta}$, where star $S_\delta$ denotes the graph $K_{1,\delta-1}$ and $mS_{\delta}$ denotes $m$ disjoint copies of the star.
	Our main results, Theorems~\ref{thm: bounds for T - spex} and \ref{thm: T-spex tighter bounds version}, are based on corresponding embedding results, Theorems~\ref{thm: max degrees is G[R]} and \ref{thm: tree embedding}, respectively.
	
	We start by first stating our  main spectral theorems. These are proved in Section~\ref{sec: spectral fixed tree}. We need to establish some notation before we can state these results.
	
	For any non-negative integer $d$ and any $d$-regular graph on $n-l$ vertices $H_d$ (not necessarily simple, but undirected), let $H_d'$ denote the graph $K_l \vee H_d$ on $n$ vertices. Then $\lambda(H_d')$ is the largest root of the quadratic polynomial   \[p_d(x) = (x- (l-1))(x - d) - l(n-l) = x^2 - (l+d-1)x + (l-1)d - l(n-l). \]
	Here, $p_d(x)$ is the characteristic polynomial of a $2$-dimensional quotient matrix $Q_d$ of $A(H_d')$ where 
	\[Q_d = \begin{bmatrix}
		l-1 & n-l\\
		l & d
	\end{bmatrix}.\]
	Thus, 
	\begin{align*}
		\lambda(Q_d) = \lambda(H_d')  &= \frac{l+d-1 + \sqrt{(l+d-1)^2 + 4(nl + d -ld -l^2)} }{2} = \sqrt{nl} +o(n^{1/2}).
	\end{align*}    
	
	For the sake of clarity and compactness, we define for a fixed $l$,\\ $f(x, n) = {(l+x-1 + \sqrt{(l+x-1)^2 + 4(nl + x -lx -l^2)}) }/{2}$. Thus, $f(d,n) = \lambda(H_d')$ for any non-negative integer $d$, and $f(d,n) = \Theta_l(n^{1/2})$.
	
	Using the structural results from the embedding Theorem~\ref{thm: max degrees is G[R]} 
	we obtain the following result on $\spex(n, T)$ and the structure of graphs in $\SPEX(n, T)$. 

	\begin{thm}
		\label{thm: bounds for T - spex}
		Let $T$ be a tree such that $T \in \mathcal{T}_{m,l + 1}^{\delta}$ for some $m,l,\delta \in \mathbb{N}$. Then, there exists $N \in \mathbb{N}$ such that for all $n > N$, if $G \in \SPEX(n, T)$, we have that $G = H_1 \vee H_2$ where $H_1$ is a graph on $l$ vertices and $H_2$ is a graph on $n-l$ vertices with $\Delta(H_2) \le \delta - 1$.

		Moreover, if $G' = K_l \vee H_2'$ is any graph on $n$ vertices where $H_2'$ is a graph on $n-l$ vertices with $\Delta(H_2') \le \delta - 1$ and at most one vertex of $H_2'$ has degree $\delta - 1$, then $\lambda(G) \ge \lambda(G')$.
		
		Consequently,
		\[f(\delta - 2, n) \le \spex(n, T) \le f(\delta-1, n).\]
	\end{thm}
	
	Note that the upper bound on $\spex(n, T)$ in Theorem~\ref{thm: bounds for T - spex} was recently proved independently in \cite[Theorem 7]{fang2024spectral} for $n$ sufficiently large.

	The bounds in Theorem~\ref{thm: bounds for T - spex} determine $\spex(n, T)$ within an interval of length $f(\delta - 1, n) - f(\delta - 2, n)$. Moreover,
	
	$
	\frac{1}{2} < f(\delta - 1, n) - f(\delta - 2, n) \\
	\hspace*{0.25cm}\le  \frac{1}{2}\left(1 + \frac{2(l+ \delta) - 1}{\sqrt{(l+\delta-2)^2 + 4(nl + \delta-1 -l(\delta-1)) -l^2)} + \sqrt{(l+\delta -3)^2 + 4(nl + \delta -2 -l(\delta -2) -l^2)}}\right) = \frac{1}{2}+ O(n^{-1/2})$.\\

	In our next theorem, we are able to show that it is possible, using the embedding Theorem~\ref{thm: tree embedding}, to determine the structure of a spectral extremal graph $G$ more precisely by giving an almost complete description of the degrees of vertices. This improves the upper bound in Theorem~\ref{thm: bounds for T - spex} to $f(\delta-2, n) + \Theta(n^{-1})$ if it is known that $T$ is embeddable in $\overline{K_l} \vee m S_{\delta}$, determining $\spex(n, T)$ within an interval of length $\Theta(n^{-1})$. These results make progress on determining $\spex(n,T)$ for $\delta \ge 2$ as asked in \cite{fang2024spectral}. 
	
	
	\begin{thm}
		\label{thm: T-spex tighter bounds version}
		Let $T$ be a tree such that $T \in \mathcal{T}_{m,l + 1}^{\delta}$ for some $m,l,\delta \in \mathbb{N}$. If $T$ is embeddable in $\overline{K_l} \vee m S_{\delta}$, then  there exists $N \in \mathbb{N}$ such that if $n > N$, then for any $G \in \SPEX(n, T)$ we have that $G = K_l \vee H_2$ where $H_2$ is a graph on $n-l$ vertices with $\Delta(H_2) \le \delta -1$ and at most $(m-1) ((\delta - 1)^2 + 1)$ vertices of $H_2$ have degree $\delta - 1$.

		Consequently,  
		\[f(\delta - 2,n) \le \spex(n, T) \le f(\delta-2,n) + \frac{2(m-1) ((\delta - 1)^2 + 1)}{n}.\]

	\end{thm}
	
	The structure and examples of trees (including caterpillars and lobsters, and all trees with $t<l$) that are embeddable in $\overline{K_l} \vee m S_{\delta}$ are discussed in Section~\ref{sec: ex for embedding}. We give a description of such embedding results when we discuss Section~\ref{sec: embeddings} results below.
	
	Note that improving the upper bound on the number of vertices of degree $\delta - 1$ in $H_2$ will improve the upper bound on $\lambda(G) = \spex(n, T)$. 
	
	\begin{remark}
		We can recover Nikiforov's results on paths \cite{Nikiforovpaths} as follows.
		For even paths $P_{2l+2}$ we have $\delta = 1$. First, Theorem~\ref{thm: T - SPEX contains k_l,n-l} implies that for $n$ sufficiently large, if $G \in \SPEX(n, P_{2l+2})$, then $K_{l, n-l} \subset G$, and therefore $G \cong H_1 \vee H_2$ for some graph $H_1$ on $l$ vertices and a graph $H_2$ on $n-l$ vertices. Theorem~\ref{thm: max degrees is G[R]}(ii) further implies that $\Delta(H_2) = 0$. Finally, since having edges in $H_1$ cannot create a $P_{2l+2}$, we must have that $G \equiv K_l \vee (n-l) K_1$. 
		For odd paths $P_{2l+3}$, we have $\delta = 2$, and  $K_{l, n-l} \subset G$ for any graph $G \in \SPEX(n, P_{2l+2})$. Therefore $G \cong H_1 \vee H_2$ for some graph $H_1$ on $l$ vertices and a graph $H_2$ on $n-l$ vertices. Theorem~\ref{thm: max degrees is G[R]}(i) gives us that $\Delta(H_2) \le 1$. Further, Theorem~\ref{thm: max degrees is G[R]}(ii) now gives us that $K_l \vee (n-l)K_1$ is $P_{2l+3}$-free, so $\lambda(G) \ge \lambda(K_l \vee (n-l)K_1)$. Since, $t= 0 < l$ for odd paths, we also know that $H_2$ has less than $m$ disjoint edges. Using routinre arguments one can pin down that exact structure of $G$ by showing that $H_2$ has exactly one vertex, thus recovering Nikiforov's result for odd paths by showing $G \cong K_l \vee (K_2 \cup (n-l-2)K_1)$.   
		
	\end{remark}

	Note that Theorem~\ref{thm: T-spex tighter bounds version} enhances Theorem~\ref{thm: bounds for T - spex} by giving us conditions for $T$ so as to guarantee that $\spex(n, T)$ is ``close" to $f(\delta - 2,n)$. The following conjecture claims that if the conditions of Theorem \ref{thm: T-spex tighter bounds version} are not met then $\spex(n, T)$ is instead ``close" to $f(\delta -1,n)$
	
	\begin{conj}
		Let $T$ be any tree in $\tmld$. If $T \in \tmld$ cannot be embedded in $\overline{K}_l \vee m S_{\delta}$, then $\spex(n, T) = f(\delta - 1,n) + O(n^{-1})$ for $n$ large enough.
	\end{conj}
	\vspace*{0.2cm}
	
	Now, we complete the discussion of the remainder of the paper.
	\vspace*{0.2cm}
	
	In Section~\ref{sec: some structural results}, we prove the basic structure of spectral extremal graphs of any $T\in\tmld$. 
	Note that the following theorem was recently proved independently in \cite[Theorem 1]{fang2024spectral}. We give an explicit bound on the threshold $N$ for $n$ as given in (\ref{threshold for n}), which is at most $m^{(12+o_l(1))l+42}$ or more simply, $m^{O(l)}$. We have not attempted to optimize this bound. This same bound on $N$ applies to Theorems~ \ref{thm: bounds for T - spex} and \ref{thm: T-spex tighter bounds version}.
	
	\begin{thm}
		\label{thm: T - SPEX contains k_l,n-l}
		Let $T$ be a tree such that $T \in \mathcal{T}_{m,l + 1}^{\delta}$ for $m,l,\delta \in \mathbb{N}$. There exists $N \in \mathbb{N}$ such that for all $n > N$, any graph $G \in \mathrm{SPEX}(n, T)$ contains a copy of $K_{l, n-l}$.
	\end{thm}
	
	It is shown in \cite{byrne2024general} that for almost  all trees $\SPEX(n, T) = K_{l, n-l}$. Note that our result applies to all trees and shows that for any tree $T$, the complete bipartite graph $K_{l, n-l}$ is a subgraph of any graph $G$ in $\SPEX(n, T)$. What remains in determining the structure of $G$ are the edges  not covered by the subgraph $K_{l, n-l}$. We explore this next.
	\vspace*{0.2cm}
	
	In Section~\ref{sec: embeddings}, our results, which might be of independent interest, consider embedding trees $T \in \tmld$ into graphs of the form $H = H_1 \vee H_2$ with $|V(H_1)| = l$. 
	
	Our first result, proved in Section~\ref{sec: proof of 1st embedding}, establishes that $\delta$ is the threshold for maximum degree of $H_2$ for whether all or none of the trees in $\tmld$ can be embedded into such an $H$.
	
	\begin{thm}
		\label{thm: max degrees is G[R]}
		Let $T$ be a tree such that $T \in \mathcal{T}_{m,l + 1}^{\delta}$ for $m,l,\delta \in \mathbb{N}$.
		Suppose $H= H_1 \vee H_2$ where $H_1$ is a graph on $l$ vertices and $H_2$ is a graph on $n-l$ vertices, where $n \ge m$.
		\begin{enumerate}[(i)]
			\item If $\Delta(H_2)\geq \delta$, then $T$ can be embedded in $H$. 
			\item If $\Delta(H_2) \le \delta -1$ and at most one vertex of $H_2$ is of degree $\delta - 1$, then $T$ cannot be embedded in $H$.
		\end{enumerate}
	\end{thm}
	
	This theorem is one of the key ingredients in proving our general spectral bounds that apply to any fixed tree $T$, as given in Theorem~\ref{thm: bounds for T - spex}.
	
	In our main embedding result, Theorem~\ref{thm: tree embedding}, we improve the previous theorem  by considering $H$ of the form $\overline{K_l} \vee m S_{\delta}$. Being able to embed $T$ in $\overline{K_l} \vee m S_{\delta}$ is a key ingredient of our stronger spectral bound in Theorem~\ref{thm: T-spex tighter bounds version}.
	
	Since not all trees in $\tmld$ can be embedded in $\overline{K_l} \vee m S_{\delta}$, we study conditions under which this is possible.
	The simplest class of such embeddable trees is those that satisfy $t<l$~\footnote{Recall our discussion on the parameter $t$, defined as $t := m-1 - (l+1)\delta$, in Remark~\ref{rem: t} as interpreting $t$ as the number of ``excess'' edges in $T$. Since the size of the smaller partite set of $T$ is $l+1$, the inequality $t< l$ is equivalent to saying that the size of the larger partite set is less than $\delta(l+1)$. Algebraically,  $t< l$ is equivalent to $m<(l+1)(\delta+1)$.}. But, we prove results that show the class of trees that are embeddable in $\overline{K_l} \vee m S_{\delta}$ is much larger than just trees with $t<l$.
	
	One such easy-to-check general condition is stated as  Hypothesis~\ref{hypothesis}. Since the statement of this hypothesis requires some new technical terminology defined in Section~\ref{sec: proof of 2nd embedding}, we postpone its statement till that section. The following theorem is proved in  Section~\ref{sec: proof of 2nd embedding}.

	\begin{thm}
		\label{thm: tree embedding}
		Let $T$ be a tree such that $T \in \mathcal{T}_{m,l + 1}^{\delta}$ for $m,l,\delta \in \mathbb{N}$. If $T$ satisfies Hypothesis~\ref{hypothesis}, then $T$ can be embedded into $\overline{K_l} \vee m S_{\delta}$.
	\end{thm}
	
	The following fact, proved in Section~\ref{sec: proof of 2nd embedding}, shows trees with $t<l$ form a subfamily of  trees that satisfy Hypothesis~\ref{hypothesis}. 
	
	\begin{fact}\label{fact}
		If $T\ \in \tmld$ with $t <l$,  then $T$ satisfies Hypothesis~\ref{hypothesis}.
	\end{fact}

	In Section~\ref{sec: ex for embedding}, see Examples~\ref{ex: tmld construction} and~\ref{ex:lobsters} for structure and examples of such trees. 
	
	By Theorem~\ref{thm: tree embedding}, it follows that $T$ with $t<l$ can be embedded into $\overline{K_l} \vee m S_{\delta}$. This result is sharp: for a given $l$ and any $t \ge l$, we can construct trees $T \in \tmld$ that cannot be embedded into $\overline{K_l} \vee m S_{\delta}$. 
	This is further generalized and discussed in a follow-up paper \cite{DKKfamily} that studies the spectral extremal problem for the family $\tmld$. 
	
	It is also true that many trees with $t \ge l$ satisfy Hypothesis~\ref{hypothesis}, and consequently can be embedded in $\overline{K_l} \vee m S_{\delta}$. We discuss this further and give a method for constructing such trees in Section~\ref{sec: ex for embedding}. In fact, we show the following theorem by giving a general construction that allows us to ``combine'' two trees, $T_i \in \mathcal{T}_{m_i,l_i+1}^{\delta}$, to create a new tree $T\in \tmld$, with $m=m_1+m_2$ and $l=l_1+l_2+1$, that is embeddable in $\overline{K_l} \vee m S_{\delta}$. 
	
	\begin{thm}
		\label{thm: t>l embeddible trees constructions}
		Consider a nonempty family of trees $\mathcal{T}_{m,l+1}^{\delta}$ with  $l\geq 1$ and $\delta>1$. There exists $T \in \mathcal{T}_{m,l+1}^{\delta}$ such that $T$ satisfies Hypothesis~\ref{hypothesis}, and consequently can be embedded in $\overline{K_l} \vee m S_{\delta}$.
	\end{thm}
	
	Note that this result applies to all $m,l,\delta$ regardless of the relation between $t$ and $l$.

	\subsection{Basic terminology and notation}
	\label{sec: notation}
	
	We give a summary of basic notation used throughout the paper. Otherwise, we follow standard graph theory terminology as given in~\cite{west2001introduction}. 
	
	Given a graph $G$ with $n$ vertices, the \emph{adjacency matrix} $A(G)$ of $G$ is an $n \times n$ matrix where the $(i, j)$-entry is equal to 1 if vertices $i$ and $j$ are adjacent, and 0 otherwise. Let $\lambda_1 \geq \lambda_2 \geq \cdots \geq \lambda_n$ be all eigenvalues of $A(G)$. The \emph{spectral radius} of $G$ is the largest eigenvalue of $A(G)$, denoted by $\lambda(G)$. Hence, $\lambda(G) = \lambda_1$.
	
	The Perron-Frobenius theorem for nonnegative matrices implies that the adjacency matrix $A(G)$ of a graph $G$ has a positive eigenvector $\mathbf{x} = [x_1, x_2, \ldots, x_n]$ corresponding to $\lambda(G)$. This eigenvector is called the \emph{normalized Perron vector} of $A(G)$ when it is normalized by making the largest entry of $\mathbf{x}$ equal to 1. A Perron entry for a vertex $v$, denoted by $x_v$ in the graph $G$, is $x_i$, where $i$ is the index corresponding to vertex $v$ in $A(G)$. We will call $\x_v$ the \emph{Perron weight} of $v$.
	
	Given a graph $G=(V,E)$,  $V'\subseteq V$, and  $E'\subseteq E$, we denote by $G[V']$, the {\em induced}\ subgraph of $G$ with vertex set given by $V'$. For any $U,W \subseteq V(G)$, let $E_{G}(U,W)$ denote the set of edges of $G[U\cup W ]$ that have one end point in $U$ and the other endpoint in $W$. We also define $e(U,W)=|E_{G}(U,W)|$. We define the graph $G[U,W]=(U\cup W, E_G(U,V))$.
	
	The degree of a vertex $v$ in the graph $G$, $d_G(v)$, is the number of edges incident to $v$ in $G$. We will drop $G$ from the subscript when the graph $G$ is clear from the context. A graph $G$ is called \emph{$k$-regular} if every vertex in the graph has the same degree $k$. An \emph{almost-regular graph} in this paper is a graph in which every vertex has the the same degree except one vertex whose degree differs by at most 1 from the degree of the rest of the vertices.

	The distance of two vertices $u$ and $v$ in a graph $G$ is denoted as $dist(u,v)$. For a vertex $v$ in a graph $G$, we use $N_G(v)$ to denote the neighbors of $v$ in $G$ and for a set $V\subset V(G)$, we define $N_G(V):= \cup_{v\in V} N_G(v)$. If the graph is clear from the context, we may drop the subscript and use $N(v)$. Similarly, we use $N_i(v)$ to denote the set of the vertices at distance $i$ from $v$ in $G$.

	Further, for any set of vertices $B$, let $B_i(v):= B \cap N_i(v)$. When the vertex $v$ is clear from context, we will also drop it from the notation we have just introduced for clarity. For example, for some vertex $v$ and set $L$, we routinely use $L_2$ to denote the vertices $L$ that are distance $2$ from $v$.

	\section{Tree embedding results and constructions} \label{sec: embeddings}

	\subsection{Proof of Theorem~\ref{thm: max degrees is G[R]}}
	\label{sec: proof of 1st embedding}

	\begin{proof}
		For the first part, we consider an arbitrary tree $T$ in $\mathcal{T}_{m,l + 1}^{\delta}$, and suppose $A$ and $B$ are its partite sets where $|A|=l+1$ and let $v\in A$ be a vertex for which $d_T(v)=\delta$. Let $v^*$ be a vertex in $H_2$ with degree $d_{H_2}(v^*)$ more than $\delta-1$. Now we embed vertex $v$ to vertex {$v^*$}, and embed vertices of $N_T(v)$ to arbitrary vertices of  $N_{H_2}(v^*)$. We also embed vertices of $B\setminus N_T(v))$ to arbitrary vertices of remaining vertices in $H_2$. Since $\delta = |N_T(v)|\leq| N_{H_2}(v^*)|$, and  $n \ge m$ this embedding so far is possible. Now let us embed vertices of $A-v$ to vertices in $H_1$. Since all the vertices of $H_1$ are adjacent to all the vertices in $H_2$ and $|A-v|=|V(H_1)|=l$ this is possible, therefore there is a copy of $T$ in $H$.
		
		For the second part, 
		assume to the contrary that there is a copy of $T$ in $H$ for some $H_2$ with $\Delta(H_2) \le \delta - 1$ and at most one vertex of degree $\delta -1$. Note that since $|A| = |V(H_1)| + 1$, so there must be at least one vertex of $A$ that is mapped into $H_2$. Let $P$ be the subset of vertices of $A$ that are mapped into $H_2$ in a copy of $T$ in $H$, and $|P|= p$, for some $1 \le p \le l+1$. 
		Now, $|N_T(P)| \ge p\delta - (p-1) = p(\delta - 1) + 1$. Since we have $\Delta(H_2) \le \delta-1$ and there is at most one vertex of degree $\delta - 1$ in $H_2$, one could map at most $\delta - 1 + (p-1)(\delta - 2) = p(\delta-2) + 1$ vertices of $N_T(P)$ into $V(H_2)$. Thus, at least $p(\delta - 1) + 1 - (p(\delta-2) + 1) = p$ vertices of $N_T(P)$ are mapped into $V(H_1)$. Since $|V(H_1)| = l$ this implies that $p \le l$. Therefore, there are $l+1-p > 0$ vertices of $A$ that must be injectively mapped into at most $l - p$ vertices of $V(H_1)$. This is not possible, giving a contradiction and there can be no copy of $T$ in $H$.  
	\end{proof}

	\subsection{Proof of Theorem~\ref{thm: tree embedding}}\label{sec: proof of 2nd embedding}

	Throughout this section we will use the notation established in the theorem statement.
	Now we will define additional notation to be used throughout the section. 
	
	\begin{definition}\label{def: ti}
		For a given tree $T\in \mathcal{T}_{m,l+1}^{\delta}$, let $A$ be the set of vertices in the smaller partite set. We denote the vertices of $A$ as $v_i$, so $A=\{v_1, v_2,..., v_{l+1}\}$. And, we define $t_i:= d_T(v_i)-\delta$ for $i=1,\ldots, l+1$.
	\end{definition} 
	
	We can think of $t_i$ as the ``excess" degree of the vertex $v_i$ beyond $\delta$. For each $v_i\in A$, we have $0\leq t_i$ since the minimum degree of vertices in $A$ is $\delta$. 
	Note that $\sum_{v_i\in J}t_i=m-1-(l+1)\delta$, which equals $t$ as defined in Remark~\ref{rem: t}.

	This leads to a core definition at the center of our statements and arguments. We will define $J'$ such that it contains all vertices with positive ``excess'' degree $t_i$, and every vertex in $A\setminus J'$ is at distance two of exactly one vertex in $J'$. This property of $J'$ is necessary for our procedure to embed $T$ into $\overline{K_l} \vee m S_{\delta}$.
	
	\begin{definition}\label{def: J'}
		For a given tree $T\in \tmld$,  let us define the set of vertices $J:=\{v_i\in A: t_i>0\}$. Starting from $J$, we iteratively define $J'$ as follows:
		
		\begin{enumerate}
			
			\item Let $J_1$ consist of all the vertices of $J$ along with all those vertices of $A\setminus J$ that are the interior vertices of a path of length more than 2 with endpoints in $A$. 
			\item Let $J_2$ be the unique smallest set containing $J_1$ such that no two vertices in $A \setminus J_2$ have a common neighbor. 
			\item Let $J'$ be a smallest set containing $J_2$ such that any vertex in $A\setminus J'$ does not have common neighbors with two or more vertices of $J'$. Note that $J'\subset A$.
		\end{enumerate}  
	\end{definition} 
	
	\begin{example}
		\label{ex for J'}
		See Figure~\ref{J'} for an illustration of the procedure given above for a specific tree with $\delta=3$, $l=11$, and $m=41$. The tree consists of black vertices which represent vertices of the set $A$ and white vertices for the remaining vertices of the tree. $J$ consists of black vertices with degree more than $\delta=3$. Next we add vertices $v_7$, $v_8$ and $v_9$ to the set $J$ to form $J_1$. In the figure this is visualized by moving the new vertices to the same level as vertices of $J$ (note that the tree is fixed and does not change). Next we choose a  vertex from $\{v_{10},v_{11}\}$ alongside the vertices of $J_1$ to obtain the set $J_2$. Finally we add vertices $v_5$ and $v_{10}$ to finish the procedure for defining $J'$. 
	\end{example}
	
	\begin{figure}
		\centering
		\resizebox{0.69\textwidth}{!}{\input{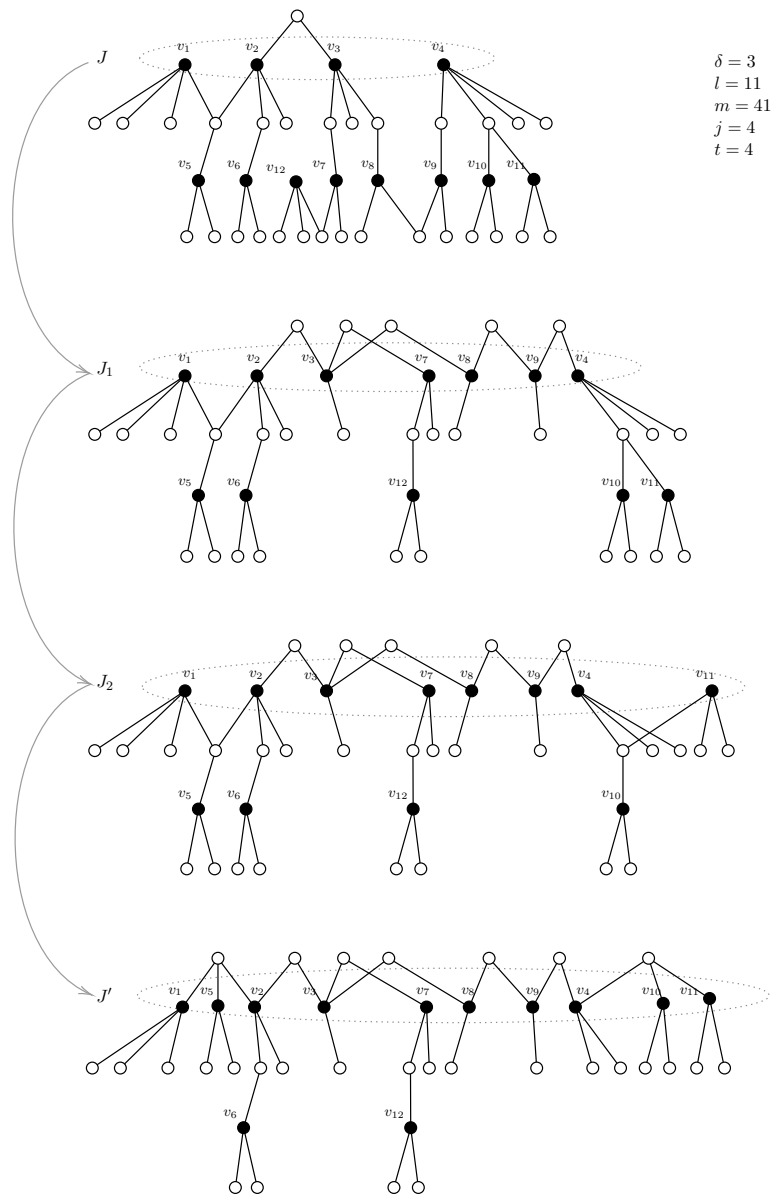}}
		\caption{An illustration of the procedure for finding $J'$.}
		\label{J'}
	\end{figure}

	Note that for the vertices in $J'$ we have  $t_i>0$  and for the vertices in $J'\setminus J$ we have $t_i=0$. For each $v\in A\setminus J'$, there is a unique vertex $v_i$ of $A$ at distance two from $v$.
	
	The set $J'$ will be central to our arguments in this section. We will now define some notation related to $J'$ that will be used in the remainder of the section. 
	
	\begin{definition}\label{def: ai}
		Let $j'=|J'|$. For each $v_i\in J'$, let $A_i$ be the set of vertices of $A\setminus J'$ that are at distance 2 from $v_i\in J' $ and let $a_i=|A_i|$. For a vertex $v=v_i\in A$ we define $t_v=t_i$, $A_v=A_i$ and $a_v=a_i$. 
	\end{definition}
	
	By the definition of $J'$, $A_i$'s partition $A\setminus J'$ and therefore $A_i$'s are pairwise disjoint and we have $\sum_{v_i\in J'} a_i = l +1- j' $. This property of $J'$ will help us in proving necessary lemmas.
	
	\begin{definition}\label{def: TI}
		For any $I\subset J'$, let $T^I$ be the induced subgraph of $T$ that has all the vertices of $I$ and the vertices of all the paths of length 2 from $I$ to $I$ (that start at a vertex in $I$ and end at a vertex in $I$). 
	\end{definition}
	See Figure~\ref{TJ} for an example of a tree $T$, the subset $J'$ of the vertices of $T$, and the corresponding subtree $T^{J'}$.

	We also know that for every $v_i\in T^{J'}$, the set of edges incident to $v_i$ in the induced subgraph $T[\{v_i\} \cup  (N_T(A_i))\cap N_T(v_i)]$ are disjoint, since $A_i\subset A\setminus J'$ . Therefore the total number of these edges must be less than or equal to the degree of $v_i$ in $T$, which equals $t_i+\delta$, so:

	\begin{equation}
		a_i +d_{T^{J'}}(v_i) \leq t_i + \delta.
		\label{e_deg}
	\end{equation}
	
	Note that $T^I$ is a forest but not necessarily a tree. In the following, we will be working with subsets $I\subseteq J'$ where $T^{I}$ is in fact a subtree of $T$. We will be identifying one of the vertices of the $T^{I}$ as its root $r$. 
	Then, for a vertex $v$ of the $T^{I}$ with the root $r$, we define $T^{I}_v$, to be the subtree of $T^{I}$ that contains $v$ and all vertices $w$ such that the path between $w$ and $r$ contains $v$. See Figure~\ref{TIV} for an example of a rooted tree $T^I$ with the given vertex $v$ and the corresponding subtree $T^I_v$.
	
	\begin{figure}[ht]
		\centering
		\resizebox{0.85\textwidth}{!}{\input{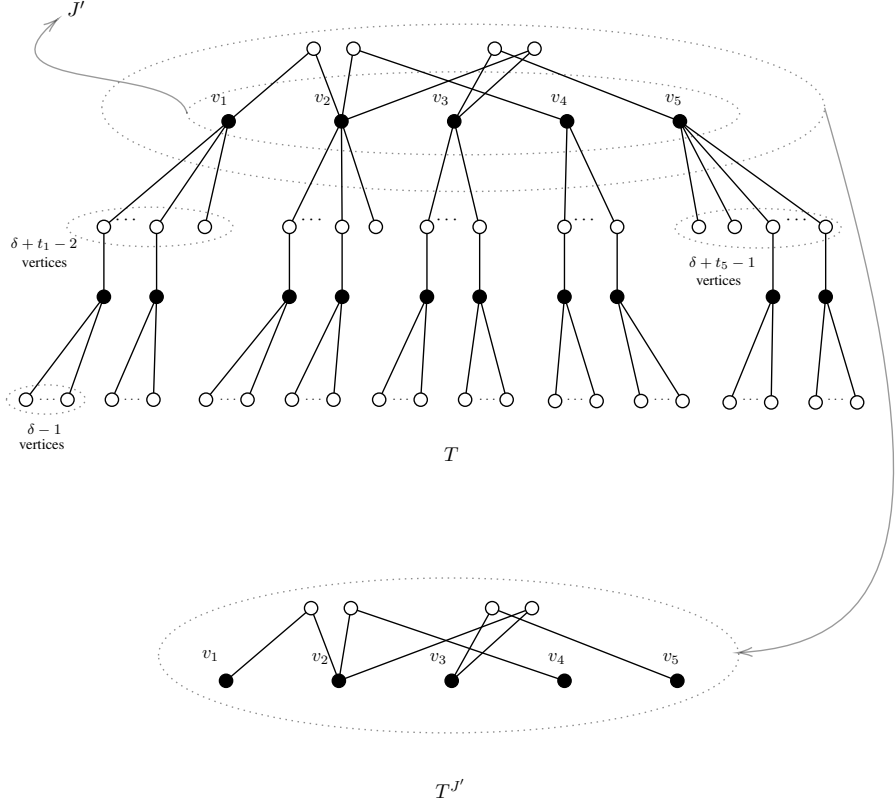}}
		\caption{An example of a tree $T$ and the corresponding subtree $T^{J'}$}
		\label{TJ}
	\end{figure}
	
	\begin{figure}
		\centering
		\resizebox{0.35\textwidth}{!}{\input{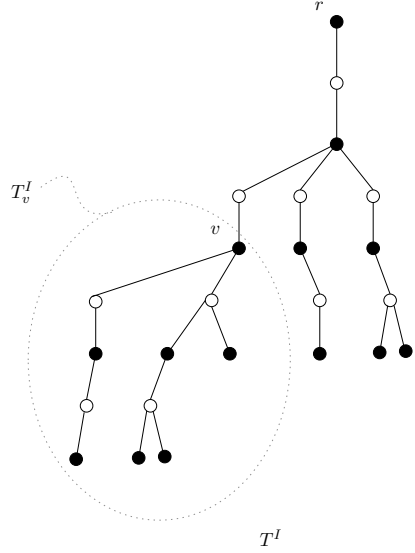}}
		\caption{A rooted tree $T^I$ with the given vertex $v$ and the corresponding subtree $T^I_v$}
		\label{TIV}
	\end{figure}

	We are now ready to state the hypothesis that is essential for our embedding arguments to work. We will show that any tree $T \in \tmld$ that satisfies this hypothesis, can be embedded in  $\overline{K}_l \vee m S_{\delta}$.
	
	Let $H_1$ be the copy of $\overline{K}_l$ and $H_2$ be the copy of $mS_{\delta}$ in $\overline{K}_l \vee m S_{\delta}$.
	In our embedding procedure, first for some set $I\subset J'$, we embed the tree $T^{I}$  such that vertices of $I$ are embedded into $H_2$ and the rest of vertices of $T^I$ into $H_1$. Then for each $v_i\in I$ we embed the vertices of $A_i$ into $H_2$. Then we embed the common neighbor of each vertex of $A_i$ and $v_i$ for all $v_i\in I$ to  $H_1$. Then we embed some neighbors of each  $v_i\in I$ into  $H_1$ such that
	the total number of vertices embedded into $H_1$ are $1+\sum_{v_i\in I}t_i$. Note that we have only embedded vertices of $V(T)\setminus A$ so far.
	To guarantee that we can embed the rest of vertices of $A$ into  $H_1$, we need the number of vertices embedded into  $H_1$ to be less than the number of vertices of $A$ that are embedded into  $H_2$ which is $\sum_{v_i\in I} a_i+ |I|$. This ensures that we have enough vertices left in  $H_1$ for the remaining vertices of $A$. This means we need the following inequality to hold for $I$,
	$$\sum_{v_i\in I} a_i+ |I|\geq 2+\sum_{v_i\in I}t_i,$$
	
	which is exactly the condition~\ref{eq:hypthesis} below.
	
	\begin{hypothesis}
		\label{hypothesis}
		
		Given $T\in \tmld$ and a non-empty set $I\subset J'$, we say $T$ satisfies the hypothesis with $I$ if $T^I$ is a tree and 
		
		\begin{equation}\label{eq:hypthesis}
			\sum_{v_i \in I} a_i \geq 2 +  \sum_{v_i \in I} (t_i - 1).
		\end{equation} 
		
		In addition, if $|I|>1$, then we also require for every $v_i\in I$ that $a_i\leq t_i$, and  $N_T(T^I\setminus J')=I$.

	\end{hypothesis}
	
	We will simply say $T$ satisfies Hypothesis~\ref{hypothesis} if there is a non-empty set $I\in J'$ where $T$ satisfies Hypothesis~\ref{hypothesis} with $I$.
	Also,  we say $I$ satisfies Property 2 if inequality (2) holds when the summation is taken over the vertices of $I$.
	
	Now we are ready to prove Fact~\ref{fact}.

	\begin{customfact}{\bf \ref{fact}}
		Let $T$ be a tree such that $T\in \tmld$ with $t<l$. Then $T$ satisfies Hypothesis~\ref{hypothesis}.
	\end{customfact}
	\begin{proof}
		
		If $J'$ has a vertex $v$ such that $a_v>t_v$ then $T$ satisfies Hypothesis~\ref{hypothesis} with $I=\{v\}$. 
		Now assume for every vertex $v_i\in J'$ we have $a_i\leq t_i$. Since $t<l$,
		\begin{equation*}
			\sum_{v_i \in J'} a_i  =l+1-j' \geq  2+t-j'=2 +  \sum_{v_i \in J'} (t_i - 1).
		\end{equation*}
		
		Since $T^{J'}$ is a tree and $N_{T^{J'}}(T^{J'}\setminus J')=J'$, so $T$ satisfies Hypothesis~\ref{hypothesis} with $I=J'$.
	\end{proof}
	
	The following lemma is the key to finding an embedding of $T$ in $\overline{K}_l \vee m S_{\delta}$.
	
	\begin{lem}
		\label{lem: there is I for T'}
		Let $T\in \tmld$ and $I\subset J'$ with $|I|>1$ such that $T$ satisfies Hypothesis~\ref{hypothesis} with $I$.
		Then there exists a non-empty set of vertices $I'\subset J'$, where $T$ satisfies Hypothesis~\ref{hypothesis} with $I'$, and for each $v_i\in I'$ we have $a_i\leq t_i+1-d_{T^{I'}}(v_i)$.
	\end{lem}
	
	\begin{proof}

		We will prove this Lemma by a step-by-step procedure that finds the subset $I'\subset I$ that satisfies Property~ \ref{eq:hypthesis} and $T^{I'}$ is a tree such that $N_T(T^{I'}
		\setminus I')=I'$ and for each $v_i\in I'$ we have $a_i\leq t_i+1-d_{T^{I'}}(v_i)$. Note that since for any $v_i\in I$  we have $a_i\leq t_i$ therefore for any $v_i\in I'$  we have $a_i\leq t_i$ since $I'$ is a subset of $I$.

		We first choose a vertex $r\in I$ as the root for $T^{I}$.
		Let $I'=I$. In each step we update $I'$ such that $I'$ still satisfies Property~\ref{eq:hypthesis} and $T^{I'}$ is a tree.
		
		\emph{Step 0}.
		If for all vertices $v_i\in I'$, we have that $a_i\leq t_i+1-d_{T^{I'}}(v_i)$. We terminate the procedure and output this $I'$.
		
		Else, go to Step 1.
		
		\emph{Step 1}.
		Let $C=\left\{ v\in I': a_v> t_v+1-d_{T^{I'}}(v) \right\}$. The set $C$ is not empty. Now let $u$ be a vertex in $C$ at maximum distance from $r$ in $T^{I'}$.
		
		Let $W$ be the set of children $w$ of $u$ such that for all $w \in W$, we have
		\begin{equation} 
			\label{eq:ineqality fo children}
			\sum_{v_i \in I_w} a_i \leq  \sum_{v_i \in I_w} (t_i - 1),
		\end{equation}
		where $I_w=J'\cap V(T^{I'}_w)$.
		
		Note that for $I'$
		we have 
		$$\sum_{v_i \in I' } a_i \geq 2 +  \sum_{v_i \in I'}(t_i - 1).$$
		Hence, for the tree $ T^{I'}\setminus \bigcup_{w\in W}T^{I'}_w$, we have $J'\cap V(T^{I'}\setminus \bigcup_{w\in W}T^{I'}_w)=I'\setminus \bigcup_{w\in W}I_w$ so
		
		$$\sum_{v_i \in I'} a_i-\sum_{w\in W}\sum_{v_i \in I_w} a_i  \geq 2 +  \sum_{v_i \in I'}(t_i - 1)-\sum_{w\in W}\sum_{v_i \in I_w} (t_i - 1), \text{ that is, }$$
		
		\begin{equation*}
			\sum_{v_i \in I'\setminus \bigcup_{w\in W} I_w} a_i \geq 2+ \sum_{v_i \in I' \setminus \bigcup_{w\in W} I_w} (t_i - 1).
		\end{equation*}

		So by removing the subtrees $T^{I'}_w$ for $w\in W$ from $T^{I'}$, we have a subtree of $T^{I'}$  where $I'\setminus (\bigcup_{w\in W} I_w)$ is the subset of vertices of this subtree lying in $J'$.
		Now, update $ I'\leftarrow I'\setminus (\bigcup_{w\in W} I_w)$. $I'$ satisfies Property~\ref{eq:hypthesis}. Note that $T^{I'}$ is a tree and $N_T(T^{I'}
		\setminus I')=I'$ since $N_T(T^I\setminus J')=I$ and $N_{T^I}(T^{I'}\setminus I')=I'$.  
		
		If  $a_{u} \leq t_{u}+1-d_{T^{I'}}(u)$ go to Step 0. Else, go to Step 2.
		
		\emph{Step 2}. Now $a_{u} > t_{u}+1-d_{T^{I'}}(u)$.
		
		Let $u$ have $q$ children $w_1,....,w_q$  in  $T^{I'}_{u}$, meaning $d_{T^{I'}_u}(u)=q$. Let $I_{w_i}=I'\cap V(T^{I'}_{w_i})$.
		
		If $u\neq r$. we have $d_{T^{I'}}(u)=q+1$. 
		
		Let $s=-a_{u}+t_{u}+1$. Since $a_{u} > t_{u}+1-q-1$ 
		and $a_{u} \leq t_{u}$, we have $1\leq s\leq q$.

		If $u= r$,  we have $d_{T^{I'}}(u)=q$.
		
		Again, let $s=-a_{u}+t_{u}+1$. Since $a_{u} > t_{u}+1-q$ and $a_{u} \leq t_{u}$ we have $1\leq s\leq q-1$.
		
		In the rest of this step, we focus on $T^{I'}_{u}$.

		Now we update $I'$ as follows.
		
		First we restrict $I'$ to $T^{I'}_{u}$, the collection of vertices in the subtree rooted at $u$ that belong to $A$, the smaller partite set of $T$. That is, updating $I' \leftarrow I' \cap V(T^{I'}_{u}) $. Note that $T^{I'}$ is a tree and $N_T(T^{I'}
		\setminus I')=I'$ since $N_T(T^I\setminus J')=I$ and $N_{T^I}(T^{I'}\setminus I')=I'$.

		Next from the subtree $T^{I'}_{u}$ we remove some children of $u$ so that we only keep $s$ children, meaning, update 
		$ I'\leftarrow I' \setminus \bigcup_{s+1\leq i\leq q} I_{w_i}$.  Note that $T^{I'}$ is a tree and $N_T(T^{I'}
		\setminus I')=I'$ since $N_T(T^I\setminus J')=I$ and $N_{T^I}(T^{I'}\setminus I')=I'$.
		
		Therefore, $u$ in $T^{I'}_{u}$ has $s$ children, so we get $a_{u}= t_{u}+1-d_{T^{I'}}(u)$. For each of the remaining $s$ children $w_i$, we have
		\begin{equation*}
			\sum_{v_i \in I_{w_i}} a_i \geq 1+ \sum_{v_i \in I_{w_i}} (t_i - 1).  
		\end{equation*}
		Consequently,
		$$\sum_{v_i \in I'} a_i- a_{u} \geq s +
		\sum_{v_i \in I'} (t_i - 1) -(t_{u}-1),
		$$
		and
		$$ \sum_{v_i \in I'} a_i- a_{u} \geq -a_{u}+t_{u}+1 +
		\sum_{v_i \in I'} (t_i - 1) -(t_{u}-1).$$
		So, $$ \sum_{v_i \in I'} a_i\geq 2 +
		\sum_{v_i \in I'} (t_i - 1). $$
		Therefore, property~\ref{eq:hypthesis} holds for $I'$. Because $u$ was a farthest vertex from $r$ in $C$, so for all vertices $v_i\in I'$, we have that $a_i\leq t_i+1-d_{T^{I'}}(v_i)$.
		Go to Step 0, the procedure ends when we enter step 0.

		Note that every time we update $I'$ in Step 1, the degree of any vertex $u$ in $T^{I'}$ remains the same or decreases, while $a_u$ and $t_u$ remain unchanged. If any vertex is used as $u$ in Step 1, it will never reappear as $u$ for any other iteration following the loop back from Step 1 to Step 0. This is because $a_{u} \leq t_{u}+1-d_{T^{I'}}(u)$ by Step 1, and choosing $u$ again in Step 1 after going to Step 0, would have required that $a_{u} > t_{u}+1-d_{T^{I'}}(u)$, which is not the case. There are at most $l+1$ iterations therefore between Steps 1 and 0, since $|A| = l+1$. In addition we go to Step 2 at any point after that we go Step 0 and terminate the procedure since at the end of Step 2  for any  vertex $v_i\in I'$, we have that $a_i\leq t_i+1-d_{T^{I'}}(v_i)$.
		
		Since $I'$ has at most  $l+1$ vertices, this process will end within $l+1$ iterations.
	\end{proof}

	\begin{figure}
		\centering
		\resizebox{0.89\textwidth}{!}{\input{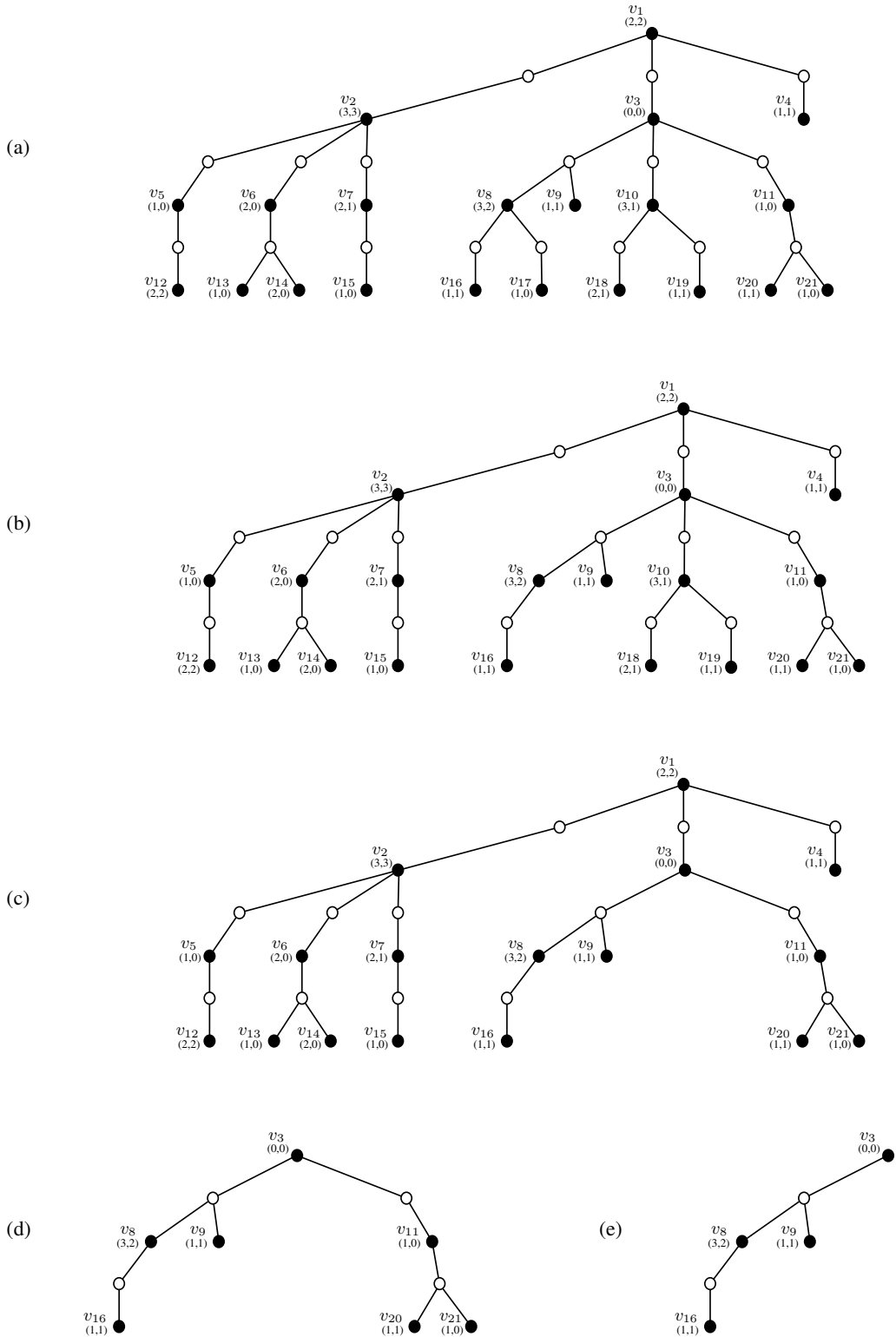}}
		\caption{An example illustrating the procedure in the proof of Lemma~\ref{lem: there is I for T'}}
		\label{fig: ex for lem}
	\end{figure}
	
	\begin{example}
		\label{ex for lem}
		In Figure~\ref{fig: ex for lem} we give an example illustrating the procedure in the proof of the Lemma~\ref{lem: there is I for T'}. 
		Let the graph in (a) be the $T^{J'}$ for $T\in \mathcal{T}_{m,l+1}^{\delta} $ where $m=185$, $l=37$ and $\delta=4$ and $t=32$. Note that since $t<l$, tree $T$ satisfies the hypothesis with $J'$ by Fact~\ref{fact} and its proof.
		The black vertices are vertices of $J'=\{v_1,v_2,\cdots,v_{21}\}$. Let the ordered pair beneath the names of vertices of $J'$ be $(t_i,a_i)$. Now we will see how the procedure in Lemma~\ref{lem: there is I for T'} works on $T^{J'}$.
		
		We choose vertex $v_1$ to be the root $r$. Then we start with $I'=J'=\{v_1,v_2,\cdots,v_{21}\}$. 
		In Step 0, since we have vertices in $I'$ with $a_i>t_i+1-d_{T^{I'}}(v_i)$ we do nothing and go to Step 1.
		In Step 1, we have $C=\{v_1,v_2,v_3,v_8\}$. The vertex $v_8$ has the maximum distance from $r=v_1$ so we choose $u$ to be $v_8$. Now we check the Inequality~\ref{eq:ineqality fo children} for each of the two children of $v_8$ and it follows that we need to remove the child of $v_8$ on the right in (a) and everything below that child, that is $v_{17}$ and its parent. We update $I'$ to $\{v_1,v_2,\cdots,v_{21}\}\setminus \{v_{17}\}$.
		
		The graph in (b) is the updated $T^{I'}$.
		Since $a_8 \leq t_8+1-d_{T^{I'}}$ in the graph in (b) we go to Step 0. Then since there are still vertices with  $a_i>t_i+1-d_{T^{I'}}(v_i)$ we go to Step 1.
		In Step 1, we get $C=\{v_1,v_2,v_3\}$. We choose $v_3$ as vertex $u$. We check  Inequality~\ref{eq:ineqality fo children} for each of the three children of $v_3$, and find that the middle child is the only one satisfying Inequality~\ref{eq:ineqality fo children}. We are now required to remove it and every vertex below it. We update $I'$ to be $(\{v_1,v_2,\cdots,v_{21}\}\setminus \{v_{17}\})\setminus\{v_{10},v_{18},v_{19}\}$. 
		
		The subtree $T^{I'}$ at this point is shown in (c).
		Since the vertex $v_3$ in the updated $T^{I'}$ still satisfies $a_3>t_3+1-d_{T^{I'}}(v_3)$, we need to go to Step 2. The vertex $v_3$ has 2 children in (c), we get that $q=2$ and $s=-a_3+t_3+1=1$. 
		
		The subtree $T^{I'}_{v_3}$ is shown in (d). 
		Then we update $I'$ to $I'=\{v_3,v_8,v_9,v_{11},v_{16}, v_{20}, v_{21}\}$.
		Since $s=1$ we only keep one of the children of $v_3$. We arbitrarily keep the left one and  we delete the right one and everything below it, as shown in (e). We update $I'$ to be $I'=\{v_3,v_8,v_9,v_{16}\}$. Now going to Step 0,  $a_i\leq t_i+1-d_{T^{I'}}(v_i)$ holds for  all vertices $v_i\in I'$. The procedure stops and we get $I'=\{v_3,v_8,v_9,v_{16}\}$ and $T^{I'}$ is the subtree in (e). This completes the example.
	\end{example}
	
	Now we are ready to complete the proof of Theorem~\ref{thm: tree embedding} by showing how to embed any tree $T\in \tmld $ that satisfies Hypothesis \ref{hypothesis} in $\overline{K}_l \vee m S_{\delta}$.

	\begin{proof}[Proof of Theorem~\ref{thm: tree embedding}]
		
		Let $H_1$ be the copy of $\overline{K}_l$ and $H_2$ be the copy of $m S_{\delta}$ in the given $\overline{K}_l \vee m S_{\delta}$.

		Since $T$ satisfies Hypothesis~\ref{hypothesis} therefore there exists some non-empty $I \subset J'$ such that $T$ satisfies Hypothesis~\ref{hypothesis} with $I$. Depending on whether $|I|=1$ or $|I|>1$ we have two cases.
		
		First let $I=\{v\}$. Therefore, by Property~\ref{eq:hypthesis} we have $a_v>t_v$. 
		We first embed the vertex $v$ to a vertex $u$ with degree $\delta-1$ in $H_2$ . Recall that $|A_v|=a_v$ and  $a_v\geq t_v+1$. 
		Let $A'_v$ be an arbitrary subset of $A_v$ of size $t_v +1$ and let $U$ be an arbitrary subset of vertices of degree $\delta-1$ in $H_2-u$ of size $t_v +1$. 
		We embed the vertices of $A'_v$ to vertices of $U$. 
		Then, we embed vertices of $N_T(v)\cap N_T(A'_v)$ to the vertices of $H_1 $. 
		Note that $|N_T(v)\cap N_T(A'_v)|=t_v+1$ which is at most $l$, the number of vertices in $H_1$ . 
		The rest of the vertices in $N_T(v)$ and $N_T(A'_v)$ are arbitrarily embedded to $N_{H_2}(u)$ and $N_{H_2}(U)$, respectively. Such an embedding is possible because $|N_T(v)\setminus N_T(A'_v)|=(t_v+ \delta)-(t_v+1)=\delta-1$ and $|N_T(A'_v)\setminus N_T(v)|=\delta-1$.
		
		We have already embedded $t_v+2$ vertices of $A$ and only $l+1-(t_v+2)$ vertices of $A$ remain to be embedded. Note that $l-(t_v+1)$ vertices of $H_1$ are still available. 
		We embed all of these remaining vertices of $A$ arbitrarily into the available vertices of $H_1$. It follows from the construction of $J'$ that the vertices in $A\setminus (A'_v\cup \{v\})$ do not have any neighbors in
		$N_T(v)\cap N_T(A'_v)$. So we can embed all the vertices in  $N_T(A\setminus (A'_v\cup \{v\}))\setminus (N_T(v)\cup N_T(A'_v))$ to the available vertices of $H_2$. This completes the embedding of $T$ in this case.

		Now we may assume that $T$ satisfies Hypothesis~\ref{hypothesis} with $I$, where $|I|>1$. By Lemma~\ref{lem: there is I for T'}, there is a non-empty set $I'$ such that $T$ satisfies Hypothesis~\ref{hypothesis} with $I'$ and for each $v_i\in I'$ we have $a_i\leq t_i+1-d_{T^{I'}}(v_i)$.
		
		Consider $T^{I'}$ to be rooted at an arbitrary vertex $r\in I'$. 
		Let $U$ be an arbitrary subset of size $|I'|$ of vertices of degree $\delta-1$ in $H_2$.
		First embed each $v_i \in I'$ to a vertex $u_i$ in $U$. 
		We order the vertices in $I'$ by their distance from $r$ in $T^{I'}$, breaking ties arbitrarily. 
		Starting at $r$ and proceeding stepwise to each vertex $v_i$ of $I$ in this order,
		we embed the children of $v_i$ (in $T^{I'}$) to arbitrary vertices of $H_1$. Note that, we start from children of $r$ and in the first step embed $d_{T^{I'}}(r)$ children to $H_1$. For the other vertices of $I'$, in each step, we embed $d_{T^{I'}}(v_i)-1$ children to $H_1$, while the parent of $v_i$ has already been embedded in a previous step.

		This completes the embedding of all the vertices of $T^{I'}$. Let $W_1$ be the set of vertices in $H_1$ used for this embedding. So $|W_1|=1+\sum_{v_i \in I'} (d_{T^{I'}}(v_i)-1)$.
		
		Next, we embed vertices of $A_i$ to arbitrary vertices with degree $\delta-1$ in $H_2\setminus U$.
		
		Let $W_2 \subset H_1\setminus W_1$ be an arbitrary subset of vertices of size $\sum_{v_i \in I'} a_i$.
		Then, for each $v_i \in I'$, embed $N_T(v_i)\cap N_T(A_i)$  to $W_2$.
		Next we will embed the vertices in $N_T(v_i)\setminus ((N_T(v_i)\cap N_T(A_i)) \cup N_{T^{I'}}(v_i))$ for each $v_i \in I'$.
		
		Note that $| (N_T(v_i)\cap N_T(A_i))\cup N_{T^{I'}}(v_i)|=|N_T(v_i)\cap N_T(A_i)|+|N_{T^{I'}}(v_i)|=a_i+d_{T^{I'}}(v_i)$. So $|N_T(v_i)\setminus ((N_T(v_i)\cap N_T(A_i)) \cup N_{T^{I'}}(v_i))|$ = $t_i+\delta-d_{T^{I'}}(v_i)-a_i$
		
		First, we embed $t_i+1-d_{T^{I'}}(v_i)-a_i$ vertices of $N_T(v_i)\setminus ((N_T(v_i)\cap N_T(A_i)) \cup N_{T^{I'}}(v_i))$ to $H_1$. 
		Since we have $a_i\leq t_i+1-d_{T^{I'}}(v_i)$ for each $v_i \in I'$, we know that  $t_i+1-d_{T^{I'}}(v_i)-a_i$ is non negative.

		Next, we embed the remaining neighbors (in $T$) of $v_i \in I'$ into the neighbors  of $u_i$ in $H_2$. Note that it is possible since there are exactly $\delta-1$ neighbors of $v_i$ that remain to be embedded.
		
		So, the number of vertices embedded into $H_1$ so far is equal to
		\[1+\sum_{v_i \in I'}(d_{T^{I'}}(v_i)-1)+\sum_{v_i \in I'} (a_i) + \sum_{v_i \in I'}( t_i+1-d_{T^{I'}}(v_i)-a_i)= 1 + \sum_{v_i \in I'} t_i,\]
		and all of these vertices are from $B$. 
		Also, the number of vertices from $A$ embedded into $H_2$ is equal to
		\[\sum_{v_i \in I'} a_i + |I'|.\] 
		Since $I'$ satisfies Property~\ref{eq:hypthesis}, $1 + \sum_{v_i \in I'} t_i<\sum_{v_i \in I'} a_i + |I'|$. So we can embed the remaining vertices of $A$, which are $l+1-\sum_{v_i \in I'} a_i-|I'|$ in number, into the remaining vertices of $H_1$ which are $l-1 - \sum_{v_i \in I'} t_i$ in number. This is possible since $l+1-\sum_{v_i \in I'} a_i-|I'|\leq l-1-\sum_{v_i \in I'} t_i$.
		
		Next, recall that for every $v_i \in I'$, we had embedded the vertices of $A_i$ into a collection of vertices of $H_2$ of degree $\delta - 1$. Vertices of $A_i$ have degree $\delta$ and so far we have embedded exactly one of the neighbors of each vertex in $A_i$ into some vertex in $H_1$. Now we embed the remaining $\delta - 1$ neighbors for any vertex $v$ in $A_i$ into the set of $\delta - 1$ vertices adjacent to the image of $v$ in $H_2$. 
		
		At the end, we embed the remaining vertices of $B$ into the available vertices of $H_2$, which is possible since $H_2$ has at least $m$ vertices.
		
		Observe that no vertices of $B$ embedded into $H_1$ are adjacent in $T$ to any vertex of $A$ that is also embedded into $H_1$. This is due to the way in which we obtained $I'$ and then embedded it into $\overline{K}_l  \vee m S_{\delta}$. In particular, the vertices of $B$ that got embedded into $\overline{K}_l $ belong to $T^{I'} \cup_{v_i \in I'} (N_T(v_i) \cap N_T(A_i))$ and the vertices of $A$ that got embedded into $\overline{K}_l $ belong to $A \setminus (I' \cup_{v_i \in I'} A_i)$. These vertices of $B$ are not adjacent to any vertices of $A \setminus (I' \cup_{v_i \in I'} A_i)$ since we have $N_T(T^{I'}\setminus J')=I'$.
		
		Thus, we have an embedding of $T$ into $\overline{K}_l  \vee m S_{\delta}$.
	\end{proof}
	
	Using the previous theorem and Fact~\ref{fact} we have the following corollary.
	\begin{cor}
		\label{cor: embeddability when t<l}
		Let $T$ be a tree such that $T\in \tmld$ with $t<l$. Then $T$ can be embedded into $\overline{K_l} \vee m S_{\delta}$.
	\end{cor}
	
	\begin{figure}[ht]
		\centering
		\resizebox{0.99\textwidth}{!}{\input{tree_preparation.tikz}}
		\caption{Setup for Example~\ref{ex for embedding}}
		\label{fig: ex for tree prepration}
	\end{figure}
	
	\begin{example}
		\label{ex for embedding}
		
		In Figure~\ref{fig: ex for tree prepration} and Figure~\ref{fig: ex for embedding} we give an example illustrating the procedure of the embedding in the proof of Theorem~\ref{thm: tree embedding}.
		
		Figure~\ref{fig: ex for tree prepration} part (a) shows a tree $T$ with $l=9$, $m=37$ and $\delta=3$.
		In part (b) we have $T^{J'}$ and in part (c), we have the $T^{I'}$ that the procedure of Lemma~\ref{lem: there is I for T'} gives us, starting at $I = J'$ (Since $t < l$, the proof of Fact~\ref{fact} gives us that $T$ satisfies Hypothesis~\ref{hypothesis} with $I = J'$).
		
		Now we show how the embedding of the tree $T$ based on the proof of the Theorem~\ref{thm: tree embedding} works.
		
		Recall that we are embedding $T$ into the graph $H_1\vee H_2$  where $H_1$ is a copy of $\overline{K}_l$ and $H_2$ is a copy of $m S_{\delta}$. This is shown in Figure~\ref{fig: ex for embedding} without the edges between $H_1$ and $H_2$, and without drawing all $m$ copies of $S_{\delta}$ in $H_2$. We use the gray colored vertices in the figure to show vertices that have not been used for the embedding yet.
		
		First we embed $T^{I'}$ in $H_1\vee H_2$ as shown in part (a). 
		
		In part (b) we embed vertices of $A_i$ to vertices of degree $\delta-1 =2$ in $H_2$, for each $v_i\in I'$ (these are the vertices $v_{10}, v_4$ and $v_3$). Next we embed the common neighbor of $v_{10}$ and $v_1$, which is $w_5$, to vertices of $H_1$. Similarly, we embed the common neighbor of $v_4$ and $v_8$, and that of $v_3$ and $v_2$.
		
		In part (c) we embed the remaining neighbors of the vertices of $A$ which have already embedded in $H_2$ (that is, we embed the remaining neighbors of the vertices of $I' \cup_{v_i \in I'}A_i $). 
		
		In part (d) we embed the rest of the vertices of $A$ to unused vertices in $H_1$, and we embed the leaves adjacent to these vertices of $A$ in $H_2$. 
		This completes the example.
	\end{example}
	
	\begin{figure}
		\centering
		\resizebox{0.85\textwidth}{!}{\input{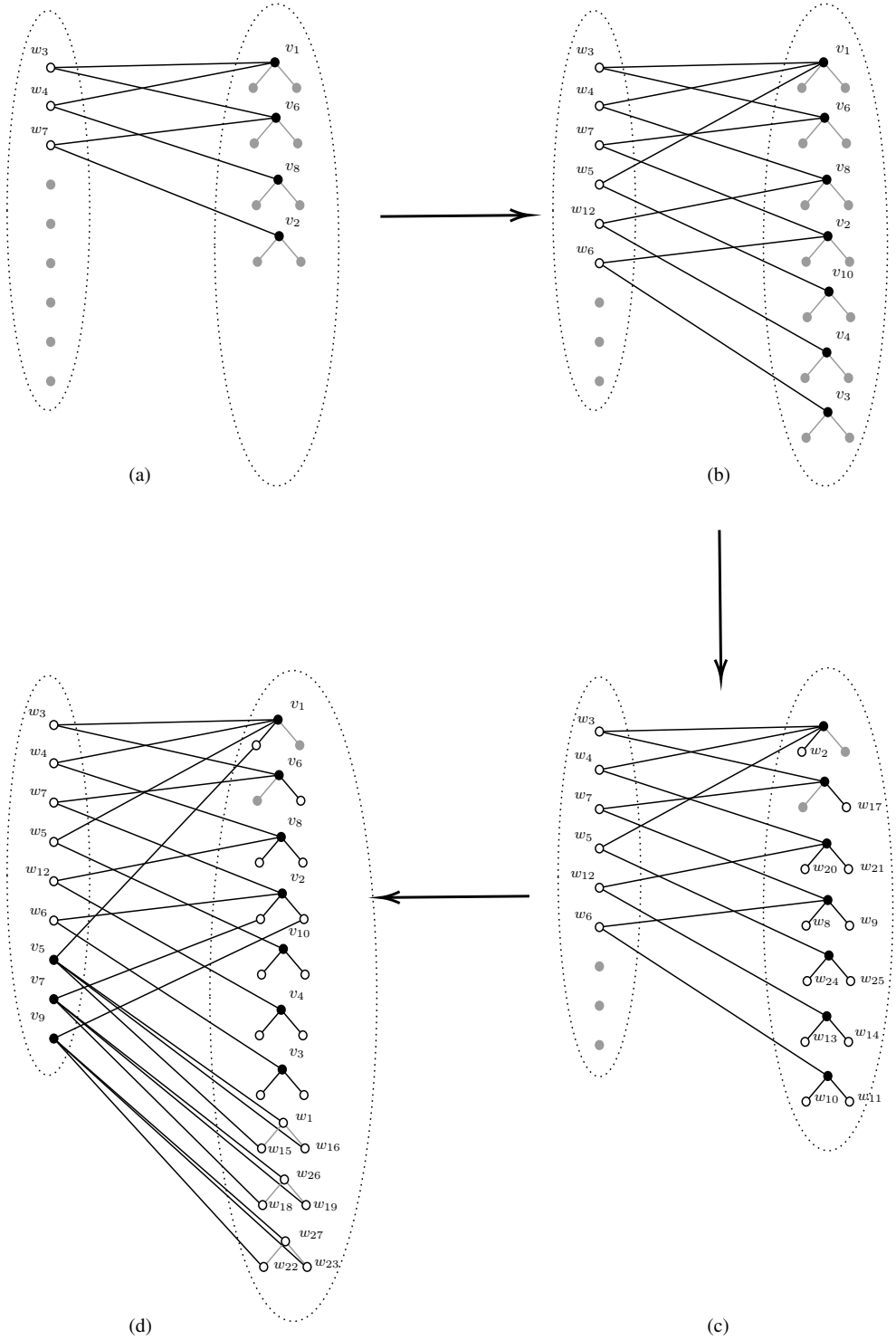}}
		\caption{An example illustrating the embedding in the proof of the Theorem~\ref{thm: tree embedding}}
		\label{fig: ex for embedding}
	\end{figure}

	\subsection{Trees embeddable in \texorpdfstring{$\overline{K}_l \vee m S_{\delta}$}{Tree embeddable}}
	\label{sec: ex for embedding}

	In this section, our aim is to illustrate the wide applicability of Theorem~\ref{thm: T-spex tighter bounds version} by describing many trees that are embeddable in $\overline{K}_l  \vee m S_{\delta}$. In the previous section, we showed that all $T\in \tmld$ that have $t<l$, or equivalently $m<(l+1)(\delta+1)$, can be embedded in  $\overline{K}_l  \vee m S_{\delta}$. First we show that there are many trees that satisfy this restriction $t <l$. Later, we will give a general construction for generating examples of $T \in \tmld$ that have $t \ge l$ and are still embeddable in  $\overline{K}_l  \vee m S_{\delta}$. 
	
	\begin{example} \label{ex: tmld construction}
		Given $m, l, \delta$, we wish to give a generic description of what a tree $T$ in $\tmld$ looks like when $t <l$. Suppose $T \in \tmld$ has the bipartition given by $\{A,B\}$ with $|A| \le|B|$. Then, $|A| = l+1$ and $|B| > (l+1)(\delta-1)$ (since $m-1 \ge (l+1)\delta$). The condition $t<l$ is equivalent to $|B| < (l+1)\delta$. By definition, $T$ has $l+1$ vertices in $A$. We describe $B$ as having two disjoint parts: $B'$ and $B''$ with $|B'|=b$ for any $b \in \{1, \ldots, l\}$ and $|B''|= m - b - (l+1)$. Vertices in $B'$ are made adjacent to vertices in $A$ in such a way that their neighborhoods overlap by at most one vertex in $A$ and $T[A\cup B']$ is connected. The vertices in $B''$ (which can number between $(l+1)(\delta -1) -b$ and $(l+1)\delta -b$) are partitioned into $l+1$ parts forming the disjoint neighborhoods in $B''$ of each vertex in $A$. To ensure the minimum degree $\delta$ of vertices in $A$, the size of any such neighborhood in $B''$ will be at least $\delta-1$ if the corresponding vertex in $A$ has one neighbor in $B'$, or at least $\delta-2$ if it has two neighbors in $B'$. Note that there is a lot of flexibility in the choice of $|B'|$, and in how the edges between $A$ and $B'$, and between $A$ and $B''$ can be chosen, which leads to many examples of trees with this overall structure. See Figure~\ref{fig: tmld construction} for some such trees. 
	\end{example}
	
	Trees in this figure, and more general examples, can also be described as lobster graphs. First, let us recall a few definitions. A \emph{caterpillar} graph is a tree that becomes a path when all its leaves are removed, and a \emph{lobster} graph is a tree that becomes a caterpillar when all its leaves are removed. Given $k \ge 1$, and $d_1, \ldots, d_k \ge 0$, $C_k(d_1, \ldots, d_k)$ denotes the family of caterpillars consisting of a path $P_k$ with $d_i$ leaves attached to the $i$th vertex for all $i\in [k]$. Note that paths, stars, double-stars, etc. can be described as caterpillars in terms of such a family.
	
	\begin{example}\label{ex:lobsters}
		Given any caterpillar $T'$, we can create a lobster $T$ from $T'$ by adding $d$ pendant edges to each leaf in $A$, where $A$ is the smaller partite set in the bipartition of $T'$, such that $t<l$ in $T$ regardless of what these values were in $T'$. We provide some simple conditions (without trying to be as general as possible) that guarantee $t<l$ in $T$. For $k$ even, let $d,d',d''$ be integers such that $1\le d \le d'' \le d' < d''+d$. Given any $T' \in C_k(d', d'', \ldots, d', d'')$, it is easy to check that $T$ has $t<l$ as $l=\frac{k}{2}(d''+1)-1$, $\delta = d+1$, and $t=\frac{k}{2}(d'-d+1)-1$. For $k$ odd, let $d,d',d''$ be integers such that $0\le d \le d'' \le d' \leq  d''+d$, and $d'' \le \lceil k/2 \rceil$. Given any $T' \in C_k(d', d'', \ldots, d')$, it is easy to check that $T$ has $t<l$ as $l=\lceil k/2 \rceil + \lfloor k/2 \rfloor d''-1$ (when $d>0$), $\delta = d+1$, and $t=\lceil k/2 \rceil(d'-d+1)-2$. Note that construction for odd $k$ subsumes other simpler families of trees with $t<l$ like odd paths ($d=d''=d'=0$),
		caterpillars based on odd paths ($d=d''=0$), etc. (see Figure~\ref{fig: tmld construction}).
	\end{example}

	\begin{figure}[ht]
		\centering
		\resizebox{0.8\textwidth}{!}{\input{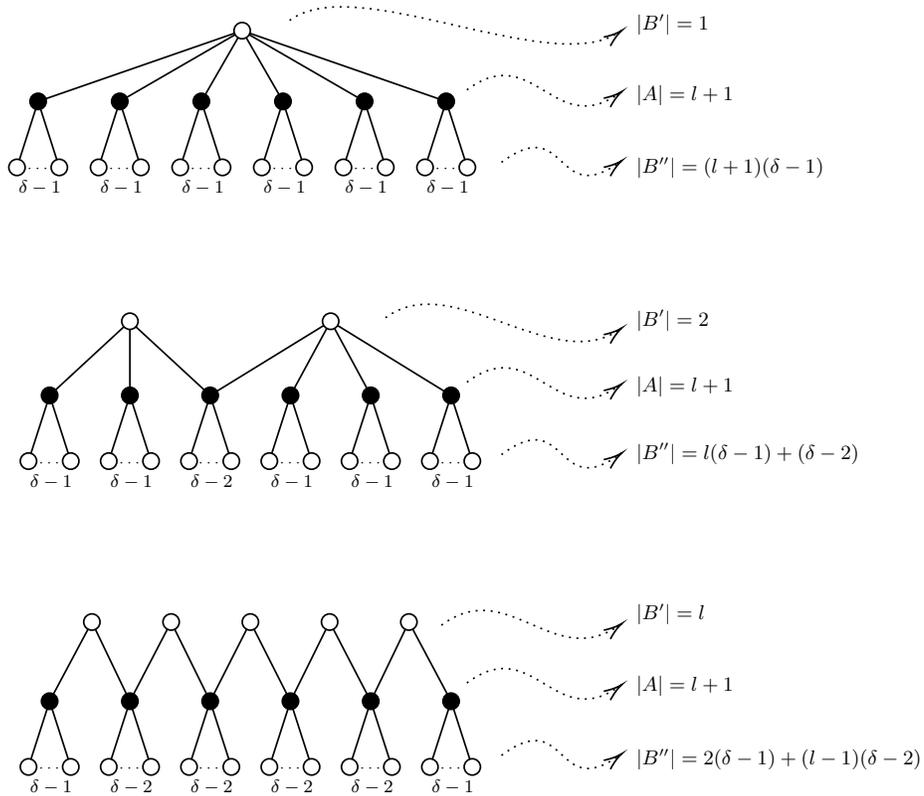}}
		\caption{Some trees in $\tmld$ with $t<l$ as described in Example~\ref{ex: tmld construction}. Note, as $t=0$ in the above examples, we can add up to $l-1$ more vertices in each of these examples while maintaining the $t <l$ condition.}
		\label{fig: tmld construction}
	\end{figure}
	
	We now give a general construction that allows us to ``combine'' two trees, $T_i \in \mathcal{T}_{m_i,l_i+1}^{\delta}$, to create a new tree $T\in \tmld$, with $m=m_1+m_2$ and $l=l_1+l_2+1$, that is embeddable in $\overline{K_l} \vee m S_{\delta}$. One of the two trees, say $T_1$, is required to satisfy $m_1<(l_1+1)(\delta+1)$, the sufficient condition for embeddability, but the other tree $T_2$ has no such restriction. By appropriate choices of $T_1$ and $T_2$, we can construct many embeddable trees including those that do not satisfy the condition $t<l$.

	Recall Remark~\ref{rem: tmld} where we gave a necessary condition, $m\geq \max\{2l+2, (l+1)\delta+1\}$, for the family $\tmld$ to be nonempty. In the next remark we show that this necessary condition is also sufficient by giving an appropriate construction that will be a useful ingredient for our general construction later.
	
	\begin{remark} \label{rem: nonempty family}
		The family of trees $\mathcal{T}_{m,l+1}^{\delta}$ is nonempty if and only if $m\geq \max\{2l+2, (l+1)\delta+1\}$  when $l>0$ and $m=\delta+1$ when $l=0$.
		
		If $l=0$, then the star $K_{1,\delta}$ has smaller partite set of size 1 and the minimum degree over vertices of smaller partite is $\delta$. 
		
		For $l>0$, we construct a tree in  consider $\mathcal{T}_{m,l+1}^{\delta}$. Let $A=\{v_1,....,v_{l+1}\}$ be a set of vertices of size $l+1$. Now we describe the set $B$ of vertices that forms a bipartition with $A$, as well as the edges in between $A$ and $B$. For each vertex $v_i$ in $A$ for $1\leq i\leq l+1$, add $\delta-1$ vertices adjacent to $v_i$  such that the neighborhoods of each $v_i$ are pairwise disjoint. Now we add a vertex adjacent to each $v_i$. To this point, we have $|A| = l+1$, $|B|=(l+1)(\delta-1) +1$, and number of edges equal to $(l+1)\delta$. Now we add $m -[(l+1)\delta+1]$ vertices adjacent to $v_1$ to the set $B$, which is possible since $m\geq (l+1)\delta+1$. The graph constructed is a tree $T$ with $m$ vertices and with bipartition $\{A, B\}$ and $|A| = l+1$, $|B|=m-(l+1)$, and number of edges equal to $m-1$. Since $A$ is of size $l+1$ and $l+1\leq m/2$, so $A$ is the smaller partite of $T$ and has $l+1$ vertices. Every vertex in the non-empty set $A\setminus \{v_1\}$ has degree $\delta$ and $v_1$ has degree at least $\delta$  and so the minimum degree over vertices of $A$, smaller partite set of $T$, is $\delta$. Therefore $T\in \mathcal{T}_{m,l+1}^{\delta}$.
		
		This completes the proof of this remark.
	\end{remark}
	
	The next lemma gives the general construction  method for ``combining'' two trees as mentioned earlier.
	
	\begin{lem}
		\label{lem: making embeddable tree}
		Suppose $T_1$ is a tree such that $T_1 \in \mathcal{T}_{m_1,l_1+1}^{\delta}$ and satisfies Hypothesis~\ref{hypothesis}, and $T_2$ is a tree such that $T_2 \in \mathcal{T}_{m_2,l_2+1}^{\delta}$, for some $m_1, l_1, m_2, l_2, \delta \in \mathbb{N}$ and $\delta>1$. Let $m=m_1+m_2$ and $l=l_1+l_2+1$. Then there is a tree $T\in \mathcal{T}_{m,l+1}^{\delta}$ with $T= T_1\cup T_2+e$ for some new edge $e$ between $T_1$ and $T_2$, such that $T$ satisfies Hypothesis~\ref{hypothesis}, and consequently can be embedded in $\overline{K_l} \vee m S_{\delta}$.
		
	\end{lem}
	
	\begin{proof}
		For a tree $T$ in $\mathcal{T}_{m,l+1}^{\delta}$, let $t_i(T)$ and $a_i(T)$  denote the quantities $t_i$ and $a_i$ respectively, as given in Definition~\ref{def: ti} and  Definition~\ref{def: ai} in Section~\ref{sec: proof of 2nd embedding}. Let $J'(T)$ be the set of vertices in $T$ as given in Definition~\ref{def: J'}.
		Let $A_1$  be the smaller partite set of $T_1$ and $A_2$  be the smaller partite set of $T_2$.
		
		First, we prove the following claim and use it for the tree $T_1$. 
		
		\begin{cl}
			\label{cl: there is a leaf}
			Let tree $T\in \mathcal{T}_{m,l+1}^{\delta}$ where $\delta>1$. 
			Then at least one of the following happens:
			\begin{enumerate}[(i)]
				
				\item There is a vertex $u$ in $J'$ which is adjacent to a leaf.
				\item There are two leaves $u_1$ and $u_2$ in $T^{J'}$ with $a_{u_i} > t_{u_i}$ for $i=1,2$.
				\item $J'=\{u\}$, and $u$ is not adjacent to a leaf and $a_u-1>t_u$.
			\end{enumerate}
		\end{cl}
		\begin{proof}

			Consider a leaf $u$ in $T^{J'}$, which means $u\in J'$ and therefore $d_{T^{J'}}(u)=1$. We show $u$ is adjacent to a leaf if $a_u\leq t_u$ . If $u$ is not adjacent to any leaf in $T$, then we have  $d_{T}(u)=t_u+\delta=d_{T^{J'}}(u)+a_u$ and since $\delta>1$, we get $a_u>t_u$ which is a contradiction.
			
			If $|J'|>1$ and there is no leaf $v$ in $T^{J'}$ where $a_v\leq t_v$, then we have two leaves $u_1$ and $u_2$ in $T^{J'}$ with $a_{u_i}> t_{u_i}$ for $i=1,2$.
			
			For $|J'|=1$, if the only vertex $u$ in $J'$ is not adjacent to any leaf, then $d_{T}(u)=t_u+\delta=d_{T^{J'}}(u)+a_u$, and since $\delta>1$ and $d_{T^{J'}}(u)=0$, we have $a_u-1>t_u$.
		\end{proof}
		
		Now we show how using a tree $T_1$ in $\mathcal{T}_{m_1,l_1+1}^{\delta}$  and a tree $T_2$ in $\mathcal{T}_{m_2,l_2+1}^{\delta}$ we can make a tree $T$ in $\mathcal{T}_{m,l+1}^{\delta}$. We will pick a vertex $w$ from $T_1$ whose existence is guaranteed by the Claim~\ref{cl: there is a leaf} and add an edge $e$ that will make this vertex adjacent to a vertex in $A_2$. Then, $T$ will be the tree $T_1\cup T_2+e$. We are only adding an edge from the bigger partite set in $T_1$ to the smaller partite set in $T_2$ so $T$ has the smaller partite set $A_1\cup A_2$ and therefore $l+1=l_1+l_2+2$ and $m=|V(T_1)\cup V(T_2)|=m_1+m_2$. 
		
		Now consider the tree $T_1$. 
		
		If $T_1$ satisfies Claim~\ref{cl: there is a leaf}(i), let $w$  be a leaf of the vertex $u$ for $T_1$, and add an edge $e$ from $w$ to a vertex in $A_2$. By our choice of $w$, in the tree $T=T_1\cup T_2+e$, $J'(T_1) = J'(T)\cap V(T_1)$ and $a_v(T)=a_v(T_1)$ and $t_v(T)=t_v(T_1)$ for $v\in J'(T_1)$. (Even though for a vertex $u$ in $A_2$, $a_u(T_2)$ and $t_u(T_2)$ might differ from $a_u(T)$ and $t_u(T)$ respectively.) 
		If a subtree $T_1$ of a tree $T$ satisfies Hypothesis~\ref{hypothesis} and its smaller partite is a subset of the smaller partite of $T$ and $N_T(T_1^I\setminus J'(T_1))=I$, then $T$ satisfies Hypothesis~\ref{hypothesis} with the same $I$ as well. Therefore we can say that, since $T_1$ satisfies Hypothesis~\ref{hypothesis}, tree $T$ satisfies Hypothesis~\ref{hypothesis}.

		$T_1^{J'_1}$ is a subtree of $T^{J'}$ and $T_1$ satisfies Hypothesis~\ref{hypothesis}, and if a subtree of a tree satisfies the hypothesis with some set $I$ where $I$ is a subset of $J'$ in $T$  then the tree satisfies the hypothesis as well, therefore $T$ satisfies the hypothesis.
		
		Suppose $T_1$ satisfies Claim~\ref{cl: there is a leaf}(ii), we have two leaves $u_1$ and $u_2$ with $a_{u_1}(T_1)>t_{u_1}(T_1)$ and $a_{u_2}(T_1)>t_{u_2}(T_2)$. Let $w$ be in $N_{T_1}(u_1)\setminus N_{T_1^{J'(T_1)}}(u_1)$. Note that $N_{T_1}(u_1)\setminus N_{T_1^{J'(T_1)}}(u_1)$ is not empty since $d_{T_1}(u_1)>1$ and $d_{T_1^{J'(T_1)}}(u_1)=1$. We add an edge $e$ from $w$ to a vertex in $A_2$ to get the tree $T$. Since  $a_{u_2}(T)=a_{u_2}(T_1)$ and $t_{u_2}(T)=t_{u_2}(T_1)$, therefore $T$ satisfies Hypothesis~\ref{hypothesis} since we have the vertex $u_2$ with $a_{u_2}>t_{u_2}$.
		
		If $T_1$ satisfies Claim~\ref{cl: there is a leaf}(iii) meaning $J'(T_1)$ consists of only one vertex $u$ which is not adjacent to any leaf in $T_1^{J'(T_1)}$ and  $a_u-1>t_u$.  Let $w$ be one of the neighbors of $u$. We add an edge $e$ from $w$ to a vertex in $A_2$ to get $T$. By the way the tree $T$ is constructed,  $J'(T_1) = J'(T)\cap V(T_1)$, $a_u(T) \ge a_u(T_1)-1$, and $t_u(T)=t_u(T_1)$. Thus, $a_{u}(T)>t_{u}(T)$ and $T$ satisfies Hypothesis~\ref{hypothesis}.
	\end{proof}
	
	We can use this lemma to show that all nonempty $\tmld$ contain a tree embeddable in $\overline{K_l} \vee m S_{\delta}$, even if $t \ge l$.

	\begin{proof}[Proof of Theorem \ref{thm: t>l embeddible trees constructions}]

		For a family of trees $\mathcal{T}_{m,l+1}^{\delta}$, the condition for not being empty is $m\geq \max\{2l+2, (l+1)\delta+1\}$. Note that since $1<\delta$, we have $\max\{2l+2, (l+1)\delta+1\}=(l+1)\delta+1$. 
		
		If $m=(l+1)\delta+1$, then since $l\geq 1$, we have $m<(l+1)(\delta+1)$. Therefore by Corollary~\ref{cor: embeddability when t<l}, it follows that any $T \in \tmld$ can be embedded in $\overline{K_l} \vee m S_{\delta}$.
		
		So it remains to consider $m>(l+1)\delta+1$.
		
		Let $y:=m-(l+1)\delta-2$. Therefore, $m=(l+1)\delta+2+y$ and $0\leq y$.
		
		Choose $l_1$ such that $1 \leq l_1 \leq l$. This is possible since $l\geq 1$.
		
		Choose $m_1$ such that $(l_1+1)\delta+1 \leq m_1 \leq \min\{(l_1+1)(\delta+1)-1,(l_1+1)\delta+1+y\}$.
		This is possible since  $1 \leq l_1$ gives us $(l_1+1)\delta+1 \leq (l_1+1)(\delta+1)-1$, and $y \geq 0$ gives us $(l_1+1)\delta+1 \leq (l_1+1)\delta+1+y$.
		
		The inequality $(l_1+1)\delta+1 \leq m_1$ guarantees that the family of trees $\mathcal{T}_{m_1,l_1+1}^{\delta}$ is non-empty and since $m_1<(l_1+1) (\delta+1)$ by the Fact~\ref{fact}, any tree $T_1\in \mathcal{T}_{m_1,l_1+1}^{\delta}$ satisfies Hypothesis~\ref{hypothesis}. 
		
		Now define $m_2= m-m_1$ and $l_2=l-1-l_1$. We want to show that $\mathcal{T}_{m_2,l_2+1}^{\delta}$ is non-empty. So we need to show $m_2\geq \max\{2l_2+2, (l_2+1)\delta+1\}=(l_2+1)\delta+1$.
		
		Since $m_1\leq (l_1+1)\delta+1+y$ and   $m=(l+1)\delta+2+y$, we have
		$m-m_1\geq (l+1)\delta+2+y- [(l_1+1)\delta+1+y]$.
		Since we have $l+1=l_1+l_2+2$, we get $m_2= m-m_1\geq (l+1)\delta+2- [(l_1+1)\delta+1]=(l_1+1+l_2+1)\delta+2- [(l_1+1)\delta+1]=(l_2+1)\delta+1$. Therefore $m_2\geq (l_2+1)\delta+1$ and consequently the family of trees $\mathcal{T}_{m_2,l_2+1}^{\delta}$ is non-empty.
		
		Now we choose trees $T_1\in \mathcal{T}_{m_1,l_1+1}^{\delta}$ and $T_2 \in \mathcal{T}_{m_2,l_2+1}^{\delta}$ which is possible since these families are non-empty. By Lemma~\ref{lem: making embeddable tree} there is $T\in \mathcal{T}_{m,l+1}^{\delta}$ that satisfies Hypothesis~\ref{hypothesis} and therefore can be embedded in $\overline{K_l} \vee m S_{\delta}$.
	\end{proof}
	
	For a non-empty family of trees $\mathcal{T}_{m,l+1}^{\delta}$ when $l=0$ , we have $l+1=1$ and since $\delta=\min\{d_T(v): v\in A \}$ where $A$ is smaller partite set of $T$, we get $m-1=\delta$ and $\mathcal{T}_{m,l+1}^{\delta}$ only contains the tree $K_{1,\delta}$ which we know cannot be embedded in $\overline{K_l} \vee m S_{\delta}$, since $\overline{K_l} \vee m S_{\delta}$ has the maximum degree $\delta-1$ while $K_{1,\delta}$ has maximum degree $\delta$.

	Now consider $T\in \mathcal{T}_{m,l+1}^{\delta}$ for $\delta=1$ and $l>0$.
	Note that $\overline{K_l} \vee m S_{\delta}$ is the complete bipartite graph $K_{l,m}$ when $\delta=1$. 
	$T$ cannot be embedded into $K_{l,m}$ since both partite sets of $T$ are of size bigger than $l$, and $T$ is connected.

	\section {On the structure of spectral extremal graphs}
	\label{sec: some structural results}

	In this section we give a series of structural lemmas in order to prove Theorem \ref{thm: T - SPEX contains k_l,n-l}.
	
	For our proof, we will need the following known facts.
	
	\begin{obs} \label{obs: 0}
		For any integer $m \geq 2$, the complete bipartite graph $K_{l+1, m}$ contains all bipartite graphs with $m$ vertices such that the smaller color class has at most $l+1$ vertices. In partitcular, $K_{l+1, m}$ contains all trees with $m$ vertices such that the smaller color class contains at most $l+1$ vertices .
	\end{obs}
	
	\begin{lem}\label{lem: evalue} 
		Let $B$ be a non-negative symmetric matrix, $y$ be a non-negative non-zero vector, and $c$ be a positive constant. If $By \geq cy$ entry-wise, then $\lambda(B) \geq c$.
	\end{lem}
	\begin{proof}
		Since $By\geq cy$ entry-wise, $y$ is a non-negative non-zero vector and $B$ is a non-negative symmetric matrix, we know that $y^TBy\geq y^Tcy$. Therefore $\lambda(B)\geq \frac{y^TBy}{y^Ty}\geq c$.
	\end{proof}
	
	The following is a well known bound on the Tur\' an number for any tree.
	
	\begin{lem}\label{lem: extree}
		For any tree $T_m$ with $m$ vertices,
		$$\frac{1}{2}(m-2)n \leq \mathrm{ex}(n,T_m) \leq (m-2)n.$$
	\end{lem}

	We define $\eta, \epsilon, \text{and } \alpha $ to be positive constants satisfying
	\begin{align}\label{choice of constants}
		\begin{split}
			\eta & < \min\left\{\frac{1}{2l}, \frac{\frac{1}{5}-\frac{1}{16lm}+\frac{1}{20l^2m}}{m-l-1}\right\}\\
			\epsilon &< \min\left\{\frac{1}{16\,l^2m}, \frac{\eta}{2}, \frac{\eta}{32\,l^2\,m+2} \right\}\\
			\alpha &< \min \left\{ \eta,\frac{\epsilon^2}{10\,m}
			\right\}.
		\end{split}
	\end{align}
	Note that $\frac{1}{2l}$ and ${(\frac{1}{5}-\frac{1}{16lm}+\frac{1}{20l^2m})}/{(m-l-1)}$ are positive since $l$ is a natural number and $m\geq 2l+2$. 
	
	Therefore, $\eta$ can be chosen. Also we can choose $\epsilon$ and $\alpha$, since they are less than the minimum of some positive numbers.
	
	In the rest of this section, we will assume that positive numbers $\eta, \epsilon$ and $\alpha$ satisfy the above constraints.

	We will now define four sets that will be used through the rest of this section. 
	
	Let $G \in \mathrm{SPEX}(n, T)$ where $T \in \tmld$. Let $L = \{v \in V=V(G) : \mathrm{x}_v \geq \alpha\}$ , $S = V \setminus L$, $M=\{v \in V=V(G) : \mathrm{x}_v \geq \alpha/(\frac{3}{2}m)\}$, and $L'=\{v \in V=V(G): \mathrm{x}_v \geq \eta\}$. The subset $L$ consists of the vertices with ``Large" Perron weight and $S$ is the set of vertices with ``Small" Perron weight. Note that $L'\subset L\subset M$ where the  vertices of $M$ may be thought of as having ``not too small" Perron weight and $L'$ may be thought of as denoting the set of vertices with ``very large" Perron weight.

	The following lemma provides upper and lower bounds for $\lambda(G)$. 
	Note that we always assume $n\geq m\geq 2l+2$. Since if $n<m$ then we have $SPEX(n,T)=K_n$.

	\begin{lem} \label{lem2**}
		Let $G \in \mathrm{SPEX}(n, T)$ for any fixed $T \in \mathcal{T}_{m,l+1}^{\delta}$. Then,
		$$\sqrt{l(n-l)} = \lambda(K_{l, n-l}) \leq \lambda(G) \leq \sqrt{(2m -4)n}.$$

	\end{lem}
	
	\begin{proof}
		
		The lower bound is derived from the spectral radius of $K_{l, n-l}$. For the upper bound, consider $A^2(G)$. Let $\mathrm{x}$ be the Perron vector of $A(G)$, normalized such that $||\mathrm{x}||_{\infty} = 1$. Let $\lambda_1 \ge \lambda_2 \ge \ldots, \lambda_{n}$ be the eigenvalues of $A(G)$, in descending order. So, $\lambda_1 = \lambda(G)$. The trace of $A^2(G)$ is equal to $\sum_{v\in V}d_G(v)=\sum_{i=1}^{n} \lambda_i^2$. Therefore,
		$\lambda_1^2 \leq \sum_{i=1}^{n} \lambda_i^2 = \sum_{v\in V}d_G(v) \leq 2e(G) \leq 2\mathrm{ex}(n, T) \leq (2m-4)n.$ 
	\end{proof}
	
	We note that one may also use Theorem \ref{spectral erdos sos} to derive the upper bound above, for $n$ sufficiently large.
	
	Next, we bound the sizes of the sets of vertices with ``large" and ``not too small" Perron weights.

	\begin{lem}
		Let $G \in \mathrm{SPEX}(n, T)$ for any fixed $T \in \mathcal{T}_{m,l+1}^{\delta}$. Then $|L| \le \frac{(4m-8)}{\alpha}\sqrt{n}$ and $|M| \le \dfrac{6m(m-2)}{\alpha}\sqrt{n}$.
	\end{lem}

	\begin{proof}
		Let $\mathrm{x}$ be the Perron vector of $A(G)$, scaled so that $||\mathrm{x}||_{\infty} = 1$. We also use $\lambda$ in place of $\lambda(G)$ for clarity. For $v\in L$, we have $\lambda \alpha \leq \lambda \mathrm{x}_v \leq d_G(v)$, and for $v\in M$, $\lambda \alpha/(\frac{3}{2}m) \leq \lambda \mathrm{x}_v \leq d_G(v)$. Thus, Lemma \ref{lem2**} gives
		\begin{align*}
			\sqrt{l (n-l)}|L|\alpha &\leq \lambda |L| \alpha \leq \lambda\sum_{v \in L} \mathrm{x}_v \leq \lambda\sum_{v \in V} \mathrm{x}_v \leq \sum_{v \in V} d_G(v) \leq 2 \mathrm{ex}(n, T) \leq (2m-4)n,
		\end{align*}
		and
		\begin{align*}
			\sqrt{l (n-l)}|M| \frac{\alpha}{\frac{3}{2}m} &\leq \lambda |M| \frac{\alpha}{\frac{3}{2}m} \leq \lambda\sum_{v \in M} \mathrm{x}_v \leq \lambda\sum_{v \in V} \mathrm{x}_v \leq \sum_{v \in V} d_G(v) \leq 2 \mathrm{ex}(n, T) \leq (2m-4)n.
		\end{align*}
		Therefore, we have $|L| \le \dfrac{2m-4}{\alpha} \dfrac{n}{\sqrt{(n-l)l}}$ and $|M| \le \dfrac{3m(m-2)}{\alpha}\dfrac{n}{\sqrt{(n-l)l}}$. Since $\frac{n}{\sqrt{(n-l)l}} \le 2\sqrt{n}$  when $n\geq 2l+2$, we have $|L| \le \frac{(4m-8)}{\alpha}\sqrt{n}$ and $|M| \le \dfrac{6m(m-2)}{\alpha}\sqrt{n}$.
	\end{proof}
	
	In the following lemma we show that every vertex of $L$ has degree linear in $n$. This will in turn give us a constant bound for the size of $L$ (recall that $m, l$, and $\delta$ are fixed).
	
	\begin{lem}
		\label{lem4.1}
		There exists $n_1 \leq \frac{200m^6}{\alpha^4}$ such that for all $n > n_1$ the following holds for $G \in \mathrm{SPEX}(n, T)$ for any fixed $T \in \mathcal{T}_{m,l+1}^{\delta}$.  Every vertex $v\in L$ has degree $d_G(v)\geq \dfrac{\alpha (n-l) }{10(2m-3)l} $. 
		Moreover, $|L|\leq\dfrac{50}{\alpha} m^2 l$.
	\end{lem}
	
	\begin{proof}
		
		Assume, for the sake of contradiction, that there exists a vertex $v \in L$ with degree $d_G(v) < \dfrac{\alpha (n-l)}{10(2m-3)l}$. Let $\mathrm{x}_v  = c \geq \alpha$. Note that all the distance sets used below, such as $N_i, L_i, M_i, S_i$, are defined with respect to the vertex $v$ (recall the definitions given in Section~\ref{sec: notation}). Using the second-degree eigenvector-eigenvalue equation, $A(G)^2 \mathrm{x}_v  = \lambda(G)^2 \mathrm{x}_v $, with respect to vertex $v$, we get:
		$$cl(n-l) \leq \lambda^2c = \lambda^2\mathrm{x}_v  = \sum_{u \sim v, w \sim u}\mathrm{x}_w \leq d_G(v)c + 2e(N_1) + \sum_{u \sim v}\sum_{w \sim u, w \in N_2}\mathrm{x}_w$$
		$$\leq (2m-3)d_G(v) + \sum_{u \sim v}\sum_{w \sim u, w \in M_2}\mathrm{x}_w + \sum_{u \sim v}\sum_{w \sim u, w \in N_2\setminus M_2}\mathrm{x}_w.$$
		This inequality implies:
		$$cl(n-l) - (2m-3)d_G(v) \leq \sum_{u \sim v}\sum_{w \sim u, w \in M_2}\mathrm{x}_w + \sum_{u \sim v}\sum_{w \sim u, w \in N_2\setminus M_2}\mathrm{x}_w.$$
		Using the fact that $d_G(v) < \frac{\alpha (n-l)}{10(2m-3)l}$, 
		
		we can write:
		\begin{equation}
			\label{eqn for lowerbnd degrees of L}
			cl(n-l) - \frac{c(n-l)}{10l} < \sum_{u \sim v}\sum_{w \sim u, w \in M_2}\mathrm{x}_w + \sum_{u \sim v}\sum_{w \sim u, w \in N_2\setminus M_2}\mathrm{x}_w.    
		\end{equation}
		By applying Lemma~\ref{lem: extree}, we obtain:
		$$\sum_{u \sim v}\sum_{w \sim u, w \in M_2}\mathrm{x}_w \leq e(N_1, M_2) \leq (m-2)(d_G(v) + |M|).$$
		Thus,
		
		$$\sum_{u \sim v}\sum_{w \sim u, w \in M_2}\mathrm{x}_w < (m-2)\left(\frac{\alpha (n-l)}{10(2m-3)l} + \frac{6m(m-2)}{\alpha}\sqrt{n}\right). $$
		It is clear that
		$$(m-2)\frac{\alpha (n-l)}{10(2m-3)l} < \frac{\alpha (n-l)}{10}.$$
		
		For $n \ge \frac{200m^6}{\alpha^4}$, we have: 
		
		\begin{equation*}
			(m-2)\left(\frac{6m(m-2)}{\alpha}\sqrt{n}\right) < \frac{\alpha (n-l)}{10}.
		\end{equation*}

		So,
		\begin{equation}
			\label{l}
			(m-2)\left(\frac{\alpha (n-l)}{10(2m-3)l} +\frac{6m(m-2)}{\alpha}\sqrt{n}\right) < \frac{2\alpha (n-l)}{10}.
		\end{equation}

		Besides:
		$$\sum_{u \sim v}\sum_{w \sim u, w \in N_2\setminus M_2}\mathrm{x}_w \leq e(N_1, N_2\setminus M_2) \frac{\alpha}{\frac{3}{2}m} \leq (m-2)n\frac{\alpha}{\frac{3}{2}m}.$$
		
		Using (\ref{eqn for lowerbnd degrees of L}) and (\ref{l}), we obtain:
		$$(l-0.3)\alpha (n-l) \le lc(n-l) - 0.1c(n-l) - 0.2\alpha (n-l)  \leq (m-2)n\frac{\alpha}{\frac{3}{2}m},  $$
		which is a contradiction because $l \geq 1$, and therefore, $\frac{m-2}{\frac{3}{2}m} < l-0.3$. Hence, $d_G(v) > \frac{\alpha (n-l)}{10(2m-3)l}$ for all $v \in L$.
		Furthermore,
		$$(m-2)n \geq e(G) \geq \frac{1}{2}\sum_{v \in L}d_G(v) \geq |L|\frac{\alpha (n-l)}{20(2m-3)l}.$$
		Therefore, we get $|L| \leq \frac{20(2m-3)(m-2)l}{\alpha}\frac{n}{n-l} \leq \frac{50}{\alpha} m^2 l$.
	\end{proof}
	
	Next, we show lower bounds for degrees of vertices in $L'$ that are growing linearly with respect to their Perron entries. Further, we deduce that there are roughly $ln$ edges between $S_1$ which is defined with respect to $z$ and vertices with distance less than 3 from $z$ with large Perron weight.
	
	\begin{lem} \label{lem: degL'}
		There exists $n_2\leq \max\{n_1, \frac{2500m^3 l^2}{\alpha^3}, \frac{l^2}{\epsilon^2}, \frac{50m^3l}{\epsilon \alpha} \binom{50m^2l/\alpha}{l}, \frac{100m^2 }{\alpha \epsilon}\}$ such that for all $n > n_2$ the following statements hold for $G \in \mathrm{SPEX}(n, T)$ for any fixed $T \in \mathcal{T}_{m,l+1}^{\delta}$.  
		\begin{enumerate}
			\item[(i)] If $v$ is a vertex in $L'$ with $\mathrm{x}_v = c$, then $d_G(v) \geq cn - \epsilon n$.
			\item[(ii)] With respect to vertex $z$, we have $(1-\epsilon)ln \leq e(S_1,\{z\}\cup L_1\cup L_2) \leq (l+\epsilon)n$.
			
		\end{enumerate}
		
	\end{lem}
	\begin{proof}
		We assume $n \geq \max\{n_1, \frac{2500m^3 l^2}{\alpha^3}, \frac{cl^2}{\epsilon^2}, \frac{50m^3l}{\epsilon \alpha} \binom{50m^2l/\alpha}{l},  \frac{200m^2 l}{\alpha \epsilon}\}$.
		
		Assume $v \in L'$ with $\mathrm{x}_v = c$. Note that all the distance sets used below, such as $N_i, L_i, S_i$, are defined with respect to the vertex $v$ (recall the definitions given in Section~\ref{sec: notation}). Now, considering the second-degree eigenvalue and eigenvector equation with respect to the vertex $v$, we obtain:
		
		\begin{equation}\label{eqn in scaling degrees}
			\begin{split}
				cl(n-l) \leq \lambda^2c &= \sum_{u \sim v}\sum_{w \sim u}\mathrm{x}_w = d_G(v)c + \sum_{u \sim v}\sum_{w \sim u, w \neq v}\mathrm{x}_w\\
				&\leq d_G(v)c + \sum_{u \in S_1}\sum_{w \sim u, w \in L_1 \cup L_2}\mathrm{x}_w\\
				&+ 2e(S_1)\alpha + 2e(L) + e(L_1, S_1)\alpha + e(N_1, S_2)\alpha.
			\end{split}
		\end{equation}
		
		Using Lemma \ref{lem: extree} for each part, we have:
		$$2e(S_1) \leq (2m-4)n.$$
		$$e(L_1, S_1) \leq (m-2)n.$$
		$$e(N_1, S_2) \leq (m-2)n.$$
		Now, applying these and Lemma~\ref{lem4.1},
		\begin{align*}
			2e(S_1)\alpha + 2e(L) + e(L_1, S_1)\alpha + e(N_1, S_2)\alpha &\leq (2m-4)n\alpha + 2 \binom{|L|}{2}\\ &+ (m-2)n\alpha + (m-2)n\alpha\\ &\leq 5m n\alpha.
		\end{align*}
		
		The last inequality holds as long as $n \ge \frac{2500m^3 l^2}{\alpha^3}$.

		Using $\alpha \leq \frac{\epsilon^2}{10m}$ from \ref{choice of constants} and (\ref{eqn in scaling degrees}), we have:
		\begin{equation}
			\label{eq9}
			cl(n-l) \leq d_G(v)c + e(S_1, L_1\cup L_2) + 5mn\alpha \leq d_G(v)c + e(S_1, L_1\cup L_2) + \frac{\epsilon^2n}{2}.
		\end{equation}
		
		Now we are ready to prove statement (i).  By contradiction, assume for $v \in L'$ with $\mathrm{x}_v = c$, we have $d_G(v) < cn - \epsilon n$.
		This implies,
		$$(l-c+\epsilon)nc{-cl^2}\leq (ln-d_G(v)-l^2)c \leq e(S_1, L_1\cup L_2) + \frac{\epsilon^2n}{2}.$$
		Since $v \in L' $, $c \geq \eta$, and $\eta \geq 2\epsilon$ by (\ref{choice of constants}), we have
		
		\begin{equation}
			\label{eq2}
			e(S_1, L_1\cup L_2) \geq (l-c)nc + \epsilon nc -cl^2 - \frac{\epsilon^2n}{2} \geq (l-1)nc + \frac{\epsilon^2n}{2}.
		\end{equation}
		The last inequality is true whenever $n \ge \frac{cl^2}{\epsilon^2}$, which is true since we take $n \ge \frac{l^2}{\epsilon}$ and $c \le 1$.

		\begin{cl}
			If $\delta := \frac{\epsilon\alpha}{50m^2l}$, then $S_1$ has at least $\delta n$ vertices with a degree at least $l$ in $G[S_1,L_1\cup L_2]$
		\end{cl}
		
		\begin{proof}
			By contradiction, assume at most $\delta n$ vertices in $S_1$ have a degree of at least $l$ in $G[S_1,L_1\cup L_2]$. Since $|S_1|\leq |N_1|=d_G(v)$ and by Lemma~\ref{lem4.1},
			
			$$e(S_1, L_1\cup L_2) < (l-1)|S_1| + |L|\delta n \leq (l-1)|S_1| + \frac{50}{\alpha} m^2 l \frac{\epsilon\alpha}{50m^2l} n \leq (l-1)(c-\epsilon)n + \epsilon n.$$
			Now by (\ref{eq2}),
			$$(l-1)nc - (l-2)n\epsilon > e(S_1, L_1\cup L_2) \geq (l-1)nc + \frac{\epsilon^2n}{2}.$$
			This leads to a contradiction and completes the proof of the claim.
		\end{proof}
		Let $D$ be the set of vertices in $S_1$ with a degree at least $l$ in $G[S_1, L_1 \cup L_2]$. We know that $|D| \geq \delta n$. 
		For any vertex in $D$, the number of options to choose a set of $l$ neighbors is at most 
		${\binom{|L|}{l}} \leq \binom{50m^2l/\alpha}{l}$. 
		Therefore, there exists a set of $l$ vertices in $L_1 \cup L_2$ that have at least $\delta n / \binom{|L|}{l} \geq \frac{\epsilon \alpha n}{50m^2l} / \binom{50m^2l/ \alpha}{l}$ common neighbors in $D$. This quantity is at least $m$ for $n \ge \frac{50m^3l}{\epsilon \alpha} \binom{50m^2l/\alpha}{l}$ and so we have $K_{l+1,m} \subseteq G[S_1, L_1 \cup L_2 \cup \{v\}]$, which contradicts Observation \ref{obs: 0}.
		
		This completes the proof of statement (i). Now we prove part (ii).
		
		To derive the lower bound, we employ (\ref{eq9}). Recall that $\mathrm{x}_z=1$ and our definitions for the distance sets $N_i, L_i, S_i$ are with respect to the vertex $z$. This gives the inequality:
		\begin{equation*}
			l(n-l) \leq d(z) + e(S_1, L_1\cup L_2) + \frac{\epsilon^2n}{2}.
		\end{equation*}
		
		Since $e(S_1,\{z\}\cup L_1\cup L_2) \geq d(z) - |L| + e(S_1, L_1\cup L_2)$, we obtain the lower bound by noticing that $\frac{\epsilon^2n}{2} + |L|+l^2 < l\epsilon n $ which holds by Lemma~\ref{lem4.1}, since $n \ge \max\{n_1,\frac{100m^2 }{\alpha \epsilon}\}$.

		To establish the upper bound, we assume towards a contradiction that for the vertex $z$, we have $e(S_1,\{z\}\cup L_1\cup L_2) > (l+\epsilon)n$. We will use Lemma~\ref{lem: extree} to derive a contradiction by showing that $K_{l+1,m}\subset G$. To this end, we prove the following claim.
		
		\begin{cl}
			If we set $\delta = \frac{\epsilon\alpha}{50m^2l}$, then, with respect to vertex $z$, there exist at least $\delta n$ vertices in $S_1$ with degree $l$ in $G[S_1,L_1\cup L_2]$.
		\end{cl}
		\begin{proof}
			
			Assume, for the sake of contradiction, that less than $\delta n$ vertices in $S_1$ have a degree of at least $l$ in $G[S_1,L_1\cup L_2]$. This implies $e(S_1,L_1\cup L_2)<(l-1)|S_1|+|L|\delta n\leq (l-1)n+\epsilon n$, which contradicts our assumption that $e(S_1,\{z\}\cup L_1\cup L_2)>(l+\epsilon)n$.
		\end{proof}
		
		Hence, there exists a subset $D\subset S_1$ with at least $\delta n$ vertices such that every vertex in $D$ has degree at least $l$ in $G[S_1,L_1\cup L_2]$. For any vertex in $D$, the number of options to choose a set of $l$ neighbors is at most $\binom{|L|}{l} \leq \binom{50m^2l/\alpha}{l}$.
		Therefore, there exists a set of $l$ vertices in $L_1 \cup L_2$ that have at least $\delta n / \binom{|L|}{l} \geq \frac{\epsilon \alpha n}{50m^2l} / \binom{50m^2l/\alpha}{l}$ common neighbors in $D$. This quantity is at least $m$ since $n \ge \frac{50m^3l}{\epsilon \alpha} \binom{50m^2l/\alpha}{l}$), and so we have $K_{l+1,m} \subseteq G[S_1, L_1 \cup L_2 \cup \{z\}]$, which leads to a contradiction. Therefore $e(S_1,\{z\}\cup L_1\cup L_2) \leq (l+\epsilon)n$.
	\end{proof}
	
	In the next lemma we will give a lower bound for the Perron weights of vertices with very large Perron weight and consequently give a lower bound for the degrees of these vertices. In addition, we show that there are exactly $l$ vertices with very large Perron weight. 
	Note that the bound on the order $n_3$ defined in the following lemma will be used for the rest of the section.
	
	\begin{lem} \label{lem: perron for L}
		There exists $n_3 \le \max\{n_2, \frac{200m^2 l}{\alpha \epsilon}\} $ such that for all $n > n_3$ the following holds for $G \in \mathrm{SPEX}(n, T)$ for any fixed $T \in \mathcal{T}_{m,l+1}^{\delta}$. For all vertices $v\in L'$, we have $d_G(v)\geq (1-\frac{1}{8l^2m})n$ and $\mathrm{x}_v\geq (1-\frac{1}{16l^2m})$. Moreover, $|L'|=l$.
		
	\end{lem}
	
	\begin{proof}
		Assume $n\geq \max\{n_2, \frac{200m^2 l}{\alpha \epsilon}\}$.
		
		If we show that $\mathrm{x}_v\geq (1-\frac{1}{16l^2m})$, then by using Lemma~\ref{lem: degL'} and $\epsilon\leq \frac{1}{16l^2m}$ from (\ref{choice of constants}), we can prove that $d_G(v)\geq (1-\frac{1}{8l^2m})n$. Consequently, for every vertex $v\in L'$ we have $d_G(v)\geq (1-\frac{1}{8l^2m})n$. Suppose $|L'| \ge l+1$, then since $n \ge \frac{200m^2}{\alpha \epsilon} \ge \frac{(l+1+m)8l^2 m}{8l^2 m - l - 1}$, there are at least $l+1$ vertices in $L'$ that have $m$ common neighbors outside $L'$, and $G$ contains a $K_{l+1,m}$, which contradicts Observation~\ref{obs: 0}. Thus, $|L'|\leq l$ as desired.
		
		Furthermore, if $|L'|\leq l-1$, then using Lemma~\ref{lem: degL'} part (ii) and applying (\ref{eq9}) to the vertex $z$, with $n \ge n_2 > \frac{(2l-2)32l^2}{16l^2 -5}$ we have
		\begin{equation*}
			l(n-l)\leq \lambda^2\leq e(S_1,L'\setminus L)+e(S_1,L_1\cup L_2)\eta +\frac{\epsilon^2n}{2}<(l-1)(n-l+1)+(l+\epsilon)n\eta+\frac{\epsilon^2n}{2}<l(n-l),
		\end{equation*}
		where the last inequality holds by $\epsilon\leq \eta/2$ and $\eta \leq 1/2l$ from (\ref{choice of constants}), leading to a contradiction.
		Thus, $|L'|=l$.
		
		To derive a contradiction, assume there is a vertex $v\in L'$ such that $\mathrm{x}_v<(1-\frac{1}{16l^2m})$. Refining (\ref{eq9}) with respect to the vertex $z$, we get:
		\begin{equation*}
			\begin{aligned}
				l(n-l)\le&\lambda^2\le e(S_1(z), \{z\} \cup L_1(z)\cup L_2(z)\setminus \{v\})+|N_1(z)\cap N_1(v)|\mathrm{x}_v+\frac{\epsilon^2n}{2}\\
				<& (l+\epsilon)n-|S_1(z)\cap N_1(v)|+|N_1(z)\cap N_1(v)|\left(1-\frac{1}{16l^2m}\right)+\frac{\epsilon^2n}{2}\\
				=&ln+\epsilon n+|L_1(z)\cap N_1(v)|-|N_1(z)\cap N_1(v)|\frac{1}{16l^2m}+\frac{\epsilon^2n}{2}.
			\end{aligned}
		\end{equation*}  
		Rearranging the above inequality and using the upper bound for $|L|$, we have $\frac{|N_1(z)\cap N_1(v)|}{16l^2m}< \epsilon n +\frac{\epsilon^2n}{2}+|L| +l^2\leq 2\epsilon n$. Which is true since  $n \ge\frac{200m^2 l}{\alpha \epsilon} > \frac{2}{\epsilon}(l^2 + \frac{50m^2 l}{\alpha})$. However, $v\in L'$, so $\mathrm{x}_v\geq \eta$ and $d_G(v)\geq (\eta-\epsilon)n$, and so $|N_1(z)\cap N_1(v)|\geq(\eta -2\epsilon)n > 32l^2m\epsilon n$ by $\epsilon< \frac{\eta}{32l^2m + 2}$, a contradiction.
	\end{proof}
	
	The common neighborhood of the vertices in $L'$ has at least $(1-\frac{1}{8lm})n $ vertices. Let $R$ be the set of vertices in this common neighborhood. Let $E$ be the set of vertices not in $L'$ and not in $R$. So $|E|\leq \frac{n}{8lm}$. We will show that $E=\emptyset$ and thus $G$ contains a large complete bipartite subgraph $K_{l, n - l}$.
	
	\begin{lem}
		\label{lem4.5}
		For all $n > n_3$ the following holds for $G \in \mathrm{SPEX}(n, T)$ for any fixed $T \in \mathcal{T}_{m,l+1}^{\delta}$. For any vertex $v\in V(G)$, the Perron weight in the neighborhood of $v$ satisfies $\sum_{w\sim v}\mathrm{x}_w\geq l-\frac{1}{16lm}$.

	\end{lem}
	\begin{proof}
		By Lemma~\ref{lem: perron for L}, if $w\in L'$we have  $\mathrm{x}_w\geq (1-\frac{1}{16l^2m})$. Therefore, If $v\in L'$
		\begin{equation*}
			\sum_{w\sim v}\mathrm{x}_w=\lambda \mathrm{x}_v\geq \lambda\left(1-\frac{1}{16l^2m}\right)\geq l-\frac{1}{16lm}.
		\end{equation*}
		If $v\in R$, then 
		\begin{equation*}
			\sum_{w\sim v}\mathrm{x}_w\geq\sum_{w\sim v, w\in L'}\mathrm{x}_w\geq l \left(1-\frac{1}{16l^2m}\right)\geq l-\frac{1}{16lm}.
		\end{equation*}
		Finally, let $v\in E$ where $E$ is the set of vertices not in $L'$ and not in $R$. If $\sum_{w\sim v}\mathrm{x}_w < l-\frac{1}{16lm}$, consider the graph $H$ obtained from $V(G)$ by deleting all edges adjacent to $v$ and adding the edges $uv$ for all $u\in L'$. Since $\sum_{w\sim v}\mathrm{x}_w < l-\frac{1}{16lm}$, we have $x^TA(H)x>x^TA(G)x$, and by the Rayleigh principle, $\lambda(H)>\lambda(G)$. However, there are no new trees $T\in\mathcal{T}_{m,l+1}$ that have isomorphic copies in $H$ but not in $G$. To observe this, assume to the contrary that $T$ is a new tree having an isomorphic copy in $H$ but not in $G$. Then $T$ has $m$ vertices $v=v_1,v_2,...,v_{m}$ and $v$ has at most $l$ neighbors in $T$ all of which lie in $L'$. Since $n \ge n_3$ the common neighborhood of vertices in $L'$ have at least $(1-\frac{1}{8lm})n>m$ vertices. So $T$ must have an isomorphic copy in $G$, where we use a vertex outside of $\{v_1, \ldots, v_m\}$ in place of $v = v_1$, a contradiction. Thus $H$ has no new trees. But since $\lambda(H) > \lambda(G)$, this contradicts the fact that $G \in \SPEX(n, T)$. Hence, $\sum_{w\sim v}\mathrm{x}_w \ge l-\frac{1}{16lm}$ for all $v \in E$, and we are done.
	\end{proof}

	\begin{lem}
		\label{lem: T - SPEX contains k_l,n-l}
		For all $n >n_3$, the following holds for $G \in \mathrm{SPEX}(n, T)$ for any fixed $T \in \mathcal{T}_{m,l+1}^{\delta}$. The set $E$ is empty and $G$ contains the complete bipartite graph $K_{l,n-l}$.
		
	\end{lem}
	\begin{proof}
		Assume to the contrary $E\neq \emptyset$.
		Any vertex $r\in R$ has $x_r<\eta$. Therefore, for any vertex $v\in E$, 
		\begin{equation*}
			\lambda \mathrm{x}_v=\sum_{u\sim v}\mathrm{x}_u=\sum_{u\sim v, u\in L'}\mathrm{x}_u+\sum_{u\sim v, u\in R}\mathrm{x}_u+\sum_{u\sim v, u\in E}\mathrm{x}_u\leq l-1+(m-l-1)\eta +\sum_{u\sim v, u\in E}\mathrm{x}_u,
		\end{equation*} 
		since $v$ can have at most $l-1$ neighbors in $L'$ and at most $m-l-1$ neighbors in $R$.
		
		From Lemma~\ref{lem4.5},
		\begin{equation*}
			\frac{\sum_{u\sim v, u\in E} \x_u}{\lambda \mathrm{x}_v}\geq \frac{\lambda \mathrm{x}_v-(l-1)-(m-l-1)\eta}{\lambda \mathrm{x}_v}\geq 1-\frac{(l-1)+(m-l-1)\eta}{l-\frac{1}{16lm}}\geq \frac{4}{5l},
		\end{equation*}
		where the last inequality comes from $\eta \leq \frac{\frac{1}{5}-\frac{1}{16lm}+\frac{1}{20l^2m}}{m-l-1}$ by (\ref{choice of constants}). Now consider the matrix $B=A(G[E])$ and vector $y:=x_{|E}$. We see for any vertex $v\in E$
		\begin{equation*}
			B\mathrm{y}_v=\sum_{u\sim v,u\in E}\mathrm{x}_u\geq \frac{4}{5l}\lambda \mathrm{x}_v=\frac{4}{5l}\lambda \mathrm{y}_v.
		\end{equation*}
		Hence, by Lemma~\ref{lem: evalue}, we have that $\lambda(B)\geq\frac{4}{5l}\lambda \geq \frac{4}{5l}\sqrt{l(n-l)}=\frac{4}{5}\sqrt{\frac{(n-l)}{l}}$. On the other hand by Lemma~\ref{lem2**} we have $\lambda(B)\leq\sqrt{(2m-4)|E|}\leq \sqrt{2m\frac{n}{8lm}}= \frac{1}{2}\sqrt{\frac{n}{l}}$ which is contradiction since $n>2l+2$.
	\end{proof}
	This completes the proof of Theorem~\ref{thm: T - SPEX contains k_l,n-l} with 
	\begin{equation}\label{threshold for n}
		N=n_3 < \max\left\{\frac{200m^6}{\alpha^4},\frac{2500m^3 l^2}{\alpha^3}, \frac{l^2}{\epsilon^2}, \frac{50m^3l}{\epsilon \alpha} \binom{50m^2l/\alpha}{l}, \frac{200m^2 l}{\alpha \epsilon}\right\}.
	\end{equation}
	
	Recall that $l \ge 1$, and since $m \ge 2l+2$, we have that $m \ge 4$.
	Setting $\eta = \frac{1}{6m}, \epsilon= \frac{1}{200m^4}, \alpha = \frac{1}{400000m^9}$, and substituting into (\ref{threshold for n}) we obtain 
	\begin{equation}
		\begin{split}
			N = n_3 &< \max\left\{10^{25}m^{42}, 10^{21} m^{32}, 10^5 m^{10}, 4 \times 10^9 m^{17} \times (2 \times 10^7 m^{12})^{l}, 10^{11} m^{16}\right\}\\ 
			&\leq  m^{(12+o_l(1))l+42} \le m^{O(l)}.          
		\end{split}
	\end{equation}
	
	Note that this upper bound on $N$ is dependent on the proof of Lemma~\ref{lem: degL'}. Further we have not attempted to optimize the bound and have replaced $l$ in (\ref{threshold for n}) throughout by $m$ except in the exponent, to get the final expression.
	
	\section{Spectral extremal theorems for a fixed tree}\label{sec: spectral fixed tree}

	\subsection{Proof of Theorem~\ref{thm: bounds for T - spex}}
	
	We now state a few lemmas which will be useful to prove the results in this section.
	\begin{lem}\cite{tait2019colin}
		Let $H_1$ be a $d$-regular graph on $n_1$ vertices and $H_2$ be a graph with maximum degree $k$ on $n_2$ vertices. Let $H$ be the join of $H_1$ and $H_2$. Define
		\[B:= \begin{bmatrix}
			d & n_2 \\
			n_1 & k
		\end{bmatrix}.\]
		Then $\lambda_1(H) \le \lambda_1(B)$ with equality if and only if $H_2$ is $k$-regular.
	\end{lem}
	
	In fact, this result implies the following lemma which appears in \cite{fang2024spectral} as follows.
	
	\begin{lem}\cite{fang2024spectral}
		\label{upper bounds on spectral radius for joins with max degrees}
		Let $H_1$ be a graph on $n_0$ vertices with maximum degree $d$ and $H_2$ be a graph on $n- n_0$ vertices with maximum degree $d'$. $H_1$ and $H_2$ may have loops or multiple edges, where loops add $1$ to the degree. Let $H = H_1 \vee H_2$. Define
		\[B:= \begin{bmatrix}
			d & n-n_0 \\
			n_0 & d'
		\end{bmatrix}.\]
		Then $\lambda_1(H) \le \lambda_1(B)$.
	\end{lem}
	Further we have equality in the above lemma if and only if $H_1$ is a $d$-regular graph and $H_2$ is a $d'$-regular graph.
	
	The proof of Theorem~\ref{thm: bounds for T - spex} can now be completed very easily with the help of Theorem~\ref{thm: T - SPEX contains k_l,n-l} and Theorem \ref{thm: max degrees is G[R]}.
	
	\begin{proof}[Proof of Theorem~\ref{thm: bounds for T - spex}]
		
		We know from Theorem~\ref{thm: T - SPEX contains k_l,n-l} that for any $T \in \tmld$, there exists some sufficiently large positive integer $N$, such that for all $n \ge N$, we have $K_{l, n-l}$ is contained in any $G \in \SPEX(n, T)$. So, $G = H_1 \vee H_2$ for some graph $H_1$ on $l$ vertices and some graph $H_2$ on $n-l$ vertices. Therefore, we may apply the first part of Theorem \ref{thm: max degrees is G[R]} to  obtain an upper bound for $\spex(n, T)$. Since, $T$ is not contained in $G$, it must be the case that $\Delta(H_2) \le \delta - 1$. Thus, Lemma~\ref{upper bounds on spectral radius for joins with max degrees} implies that $\spex(n, T) = \lambda(G) \le f(\delta -1, n)$, and this upper bound is realized only when $H_1 = K_l$ and $H_2$ is a $\delta - 1$ regular graph.

		To observe the lower bound, we will use second part of Theorem \ref{thm: max degrees is G[R]}. For any graph $H_2'$ on $n-l$ vertices,
		satisfying $\Delta(H_2') \le \delta -1$ and at most one vertex of $H_2'$ has degree $\delta -1$, $G' = K_l \vee H_2'$ does not contain any copy of $T$. In particular, this is the case if $H_2'$ is a $\delta - 2$ regular graph, or a graph with one vertex of degree $\delta -1$ and every other vertex of degree $\delta -2$. At least one of these is possible depending on the parity of $n-l$ and $\delta - 2$. Therefore, it must be the case that $f(\delta - 2, n) \le \lambda(K_l \vee H_2') = \lambda(G') \le \lambda(G) = \spex(n, T)$, proving the lower bound. 
	\end{proof}
	
	\subsection{Proof of Theorem~\ref{thm: T-spex tighter bounds version}}

	In this section we will consider trees $T \in \tmld$  which can be embedded in $\overline{K_l} \vee mS_{\delta}$ (which include trees that satisfy Hypothesis~\ref{hypothesis} and those trees with $t<l$).
	So far we know using Theorems~\ref{thm: T - SPEX contains k_l,n-l} and \ref{thm: max degrees is G[R]} that for any tree $T$ in $\tmld$ there exists some $N \in \mathbb{N}$, as given in Equation~(\ref{threshold for n}), such that for all $n > N$, any $G \in \spex(n, T)$ has the form $G = H_1 \vee H_2$, where $H_1$ is a graph on $l$ vertices and $H_2$ is a graph on $n-l$ vertices with $\Delta(H_2) \le \delta - 1$. 
	Before we prove Theorem~\ref{thm: T-spex tighter bounds version}, we 
	will prove some supporting lemmas that are applied to estimate the Perron entries of a spectral extremal graph. In fact, the lemmas are more general and give estimations for any graphs that have structure similar to that of the spectral extremal graphs, that is they can be obtained as the join of a graph on $l$ another graph on $n-l$ vertices, where the graph on $n-l$ vertices has maximum degree at most $d$ for some constant $d$.
	
	\begin{lem}
		\label{lem: max perron entries in R}
		
		Let $H = H_1 \vee H_2$ where $H_1$ is a graph on $l$ vertices and $H_2$ is a graph on $n-l$ vertices with $\Delta(H_2) = d$. Let $\mathrm{y}$ be a Perron vector of $H$ with $\supnorm{y} = 1$, then \[M = \max\{\mathrm{y}_v \mid v \in V(H_2)\} \le \frac{l}{\lambda(H) - d} \le \frac{l}{\sqrt{l(n-l)} - d}.\]
		
		Consequently, if  $T \in  \tmld$ ,  there exists some $N \in \mathbb{N}$, as given in (\ref{threshold for n}), such that for all $n > N$, any $G \in \SPEX(n, T)$ is of the form $G= H_1 \vee H_2$, where $H_1$ is a graph on $l$ vertices and $H_2$ is a graph on $n-l$ vertices, and $\x$ is the Perron vector of $G$ scaled so that $||\x||_{\infty} = 1$, we have
		\[\max\{\x_v \mid v \in V(H_2)\} \le \frac{l}{\lambda(G) - d} \le \frac{l}{\sqrt{l(n-l)} + 1 -\delta}.\]
	\end{lem}
	
	\begin{proof}
		Let $M = \max\{y_v \mid v \in V(H_2)\}$. Say $w \in V(H_2)$ is a vertex such that $y_{w} = M$. Then $\lambda(H) M =\lambda(H)y_w = \sum_{u \in V(H_1)} y_u + \sum_{u\sim w, u \in V(H_2)} y_u \le l + d M$. Since $K_{l, n-l} \subset H$, we have $\lambda(H) \ge \sqrt{l(n-l)}$ and therefore $M \le \frac{l}{\lambda(H) - d} \le \frac{l}{\sqrt{l(n-l)} -d}$. 
		
		The second statement follows from Theorem~\ref{thm: T - SPEX contains k_l,n-l}.
	\end{proof}
	
	\begin{lem}
		\label{lem: difference of perron entries from R}
		Let $H = H_1 \vee H_2$ where $H_1$ is a graph on $l$ vertices and $H_2$ is a graph on $n-l$ vertices with $\Delta(H_2) = d$. Let $\mathrm{y}$ be a Perron vector of $H$ with $\supnorm{y} = 1$. Let $M := \mathrm{max}\{\mathrm{y}_v : v \in R\}$ and $m := \mathrm{min}\{\mathrm{y}_v : v \in R\}$. Then,
		$M - m \le \frac{d M}{\lambda(H)} \le \frac{d M}{\sqrt{l(n-l)}}$.
		Consequently, 
		\[m \ge M - \frac{dM}{\lambda(H)} \ge M\left(1 - \frac{d}{\sqrt{l(n-l)}}\right).\]
	\end{lem}
	
	\begin{proof}
		Let $\lambda= \lambda(H) \ge \sqrt{l(n-l)}$. Let $u, w \in V(H_2)$ such that $\mathrm{x}_u = m$ and $\mathrm{x}_w = M$, respectively. Then,
		
		\begin{equation*}
			\begin{aligned}
				\lambda(M-m) = \lambda(\mathrm{x}_w - \mathrm{x}_u) &= \sum_{\substack{v \sim w\\ v\in L}}\mathrm{x}_v + \sum_{\substack{v \sim w\\ v\in R}}\mathrm{x}_v - 
				\sum_{\substack{v \sim u\\ v\in L}}\mathrm{x}_v - \sum_{\substack{v \sim u\\ v\in R}}\mathrm{x}_v\\ 
				&= \sum_{\substack{v \sim w\\ v\in R}}\mathrm{x}_v - \sum_{\substack{v \sim u\\ v\in R}}\mathrm{x}_v \\
				&\le d M.
			\end{aligned}
		\end{equation*}
		Thus, $M-m \le \frac{d M}{\lambda} \le \frac{d M}{\sqrt{l(n-l)}}$.
		
		It follows from here that $m \ge M - \frac{dM}{\lambda(H)} = M\left(1 - \frac{d}{\lambda(H)}\right) \ge M\left(1 - \frac{d}{\sqrt{l(n-l)}}\right)$.
	\end{proof}

	Next we give an upper bound for the spectral radius of any graph with looks like the join of a graph on $l$ vertices with another graph on $n-l$ vertices, where the graph on $n-l$ vertices has maximum degree at most $d$, and at most $c$ vertices of the graph on $n-l$ vertices have degree $d$.
	\begin{lem}
		\label{lem: upper bound on spectral radius for Kl join H2}
		
		Let $H = K_l \vee H_2$ where $V(H_2) = n-l$, $\Delta(H_2) = d$ and there are at most $c$ vertices in $H_2$ with degree $d$. Then $\lambda(H) \le f(d-1,n) + \frac{2c}{n}$.
	\end{lem}
	\begin{proof}
		By our assumptions, $\Delta(H_2) = d$, and there are $c$ vertices of degree $d$ in $H_2$, and therefore there exists a set $E'$ of at most $c$ edges in $E(H_2)$ that can be deleted to obtain a graph $ H_2'$ with $\Delta(H_2') \le d-1$. 
		Consequently, $\lambda(K_l \vee H_2') \le f(d-1)$.
		Let $\mathrm{y}$ be a Perron vector of $H$ with $\supnorm{y} = 1$ and set $M := \mathrm{max}\{\mathrm{y}_v : v \in R\}$ and $m := \mathrm{min}\{\mathrm{y}_v : v \in R\}$.
		Since $E(H) = E(K_l \vee H_2') \cup E'$, we know that \[ \lambda(H) = \max_{\x \ne 0} \frac{\sum_{uv \in E(H) \setminus E'} 2 \x_u \x_v + \sum_{uv \in E'} 2 \x_u \x_v}{\sum_{u \in V(H) \setminus V(H_2)}x_u^2 + \sum_{u \in V(H_2)}x_u^2}.\]
		Therefore, by Lemmas \ref{lem: max perron entries in R} and \ref{lem: difference of perron entries from R}, we can see that
		
		\begin{equation*}
			\begin{aligned}
				\lambda(H) \le f(d-1,n) + 2c\frac{M^2}{l + (n-l)m^2} &\le f(d-1,n) + 2c\frac{M^2}{l + (n-l)M^2\left(1 - \frac{d}{\lambda(H)}\right)^2}\\
				&= f(d-1,n) + 2c\frac{M^2}{M^2\left(\frac{l}{M^2} + (n-l)\left(1 - \frac{d}{\lambda(H)}\right)^2\right)}\\ 
				&\le f(d-1,n) + 2c \frac{1}{\frac{l}{\left(\frac{l}{\lambda(H) - d}\right)^2} + (n-l)\left(1 - \frac{d}{\lambda(H)}\right)^2}\\
				&\le f(d-1,n) + 2c \frac{1}{\frac{\left(\lambda(H) - d\right)^2}{l} + (n-l)\left(\frac{\lambda(H) - d}{\lambda(H)}\right)^2}\\
				&\le f(d-1,n) +  \frac{2c}{n}.
			\end{aligned}
		\end{equation*}
	\end{proof}
	
	We have for $T \in\tmld$ that is embeddable in $\overline{K_l} \vee m S_{\delta}$, any  $G \in \SPEX(n, T)$ for $n > N$ must have $G \cong H_1 \vee H_2$ where $H_1$ is a graph on $l$ vertices, and $H_2$ is $m S_{\delta}$-free. We will now prove a lemma which will give us an upper bound on the number of vertices of degree $\delta - 1$ in $H_2$.

	\begin{lem}
		\label{lem: max number of vertices of degree d in H_2}
		Let $H$ be a $k S_{d+1}$-free graph on $n -l$ vertices with $\Delta(H) \le d$. Then $H$ has at most $(k-1)(d^2 + 1)$ vertices of degree $d$. 
	\end{lem}
	
	\begin{proof}
		
		Suppose $k' \le k - 1$ is the largest number of disjoint stars $S_{d+1}$ contained in $H$. Let $c_1, c_2, \ldots, c_{k'}$ be the centers of these stars. Then any vertex of degree $d$ in $H$ must be at distance at most $2$ from one of the $c_i$, $1 \le i \le k'$. For any fixed $i$, $1 \le i \le k'$, there are $d$ vertices adjacent to each $c_i$, and each neighbor of $c_i$ is adjacent to at most $d -1$ other vertices apart from $c_i$. Hence, there are at most $1 + d + d(d-1) = d^2 + 1$ vertices at distance no more than $2$ from $c_i$. Consequently, there are at most $k'(d^2 + 1)$ vertices of degree $d$ in $H$. 
	\end{proof}

	We are now ready to complete the proof of Theorem \ref{thm: T-spex tighter bounds version}.

	\begin{proof}[Proof of Theorem \ref{thm: T-spex tighter bounds version}]
		Let $T$ be a fixed tree in $\tmld$ that is embeddable in $\overline{K_l} \vee m S_{\delta}$ . By Theorems~\ref{thm: T - SPEX contains k_l,n-l} and \ref{thm: max degrees is G[R]}, we know that there exists an $N \in \mathbb{N}$, such that for all $n > N$, any $G \in \SPEX(n, T)$ is of the form $G = H_1 \vee H_2$ where $H_1$ is a graph on $l$ vertices and $H_2$ is a graph on $n-l$ vertices with $\delta - 2 \le \Delta(H_2) \le \delta - 1$. 
		Further, by Theorem~\ref{thm: tree embedding} and Lemma~\ref{lem: max number of vertices of degree d in H_2} we know that there are at most $(m-1)((\delta- 1)^2 +1)$ vertices of degree $\delta - 1$ in $H_2$. Thus, applying Lemma~\ref{lem: upper bound on spectral radius for Kl join H2} with $d= \delta - 1$ and $c = (m-1)((\delta- 1)^2 +1)$ we have that, $\lambda(G) = \spex(n, T) \le f(\delta - 2, n) + \frac{2 (m-1) ((\delta - 1)^2 + 1)}{n}$.
		
		Also, the lower bound $f(\delta-2, n) \le \spex(n, T)$ follows from Theorem~\ref{thm: bounds for T - spex}.
		
		Finally, we end by proving that $H_1$ is isomorphic to $K_l$.
		
		Assume to the contrary that $H_1 \not\cong K_l$. Then there must exist at least one non-adjacent pair of vertices in $H_1$. Say $u_0 v_0$ is one such non-adjacent pair. We know  $H_2$ has less than $(m-1)\dmoore$ vertices of degree $\delta - 1$ and therefore we can delete fewer than $(m-1)\dmoore$ number of edges to obtain a subgraph with maximum degree $\delta -2$.
		Let $H_2'$ be the graph obtained from $H_2$ after deleting as few edges as possible from $H_2$ to obtain a graph with maximum degree $\delta - 2$. Let $G' = K_l \vee H_2'$. Then, by Theorem~\ref{thm: max degrees is G[R]}, $G'$ contains no copies of $T$. Further,  by Lemmas~\ref{lem: perron for L} and \ref{lem: max perron entries in R}
		\begin{equation*}
			\begin{aligned}
				\lambda(G') - \lambda(G) &\ge \frac{\x^T\left(A(G') - A(G)\right)\x}{\x^T \x} \ge 2\frac{\x_{u_0} \x_{v_0} - \sum_{uv \in E(H_2) \setminus E(H_2')} \x_u \x_v}{\x^T \x} \\
				&\ge 2\frac{(1-\frac{1}{16l^2 m})^2 - m\dmoore \left(\frac{l}{\sqrt{l(n-l)} + 1 -\delta}\right)^2}{\x^T \x}\\
				& >0,
			\end{aligned}
		\end{equation*}
		for $n \ge 4\delta^2 m^2 l^2$ which is true $n>N$ as given in (\ref{threshold for n}).
		
		Thus, $\lambda(G') > \lambda(G)$, which contradicts the fact that $G \in \SPEX(n, T)$. Hence, $H_1$ must be isomorphic to $K_l$.
	\end{proof}
	
	Note that improving the upper bound on the number of vertices of degree $\delta - 1$ in $H_2$ will improve the upper bound on $\lambda(G) = \spex(n, T)$. As mentioned earlier, structure and examples of trees that are embeddable in $\overline{K_l} \vee m S_{\delta}$, including trees that satisfy $t<l$ or more generally Hypothesis~\ref{hypothesis}, are discussed in Section~\ref{sec: ex for embedding}. 
	
	\vspace{0.5cm}
	\textbf{Acknowledgements:} We thank  Michael Tait for his advice during the preparation of the paper. We also thank the anonymous referee for their helpful suggestions.

\end{document}